\documentclass[11pt,a4paper]{amsart}
\usepackage[english]{babel}
\usepackage{enumerate}
\usepackage{amsmath,amssymb}
\usepackage[dvips]{graphicx}
\usepackage[latin1]{inputenc}

\usepackage{wasysym}
\usepackage{amssymb,amsfonts, color,epsf}
\usepackage{verbatim}
\usepackage{fullpage}
\usepackage{young}
\usepackage{youngtab}
\ifx\pdfpageheight\undefined
   \usepackage[dvips,colorlinks=true,linkcolor=blue,citecolor=red,%
      urlcolor=green]{hyperref}
   \usepackage[dvips]{graphicx}
   \makeatletter
   \edef\Gin@extensions{\Gin@extensions,.mps}
   \makeatother
\else

   \usepackage[bookmarksopen=false,pdftex=true,breaklinks=true,%
      backref=page,pagebackref=true,plainpages=false,%
      hyperindex=true,pdfstartview=FitH,colorlinks=true,%
      pdfpagelabels=true,colorlinks=true,linkcolor=blue,%
      citecolor=red,urlcolor=green,hypertexnames=false%
      ]%
   {hyperref}
\fi
\usepackage{bm}

\usepackage{amsfonts,mathrsfs}
\usepackage{bbm,amsmath,amssymb,amsfonts,graphicx,color,subfig,mathtools, verbatim}

\usepackage{amsmath}
\usepackage{tabularx}
\usepackage{color, colortbl}
\usepackage{url}
\makeindex

\numberwithin{equation}{section}
\newtheorem{thm}{Theorem}
\newtheorem{prop}[thm]{Proposition}
\newtheorem{lemma}[thm]{Lemma}
\newtheorem{cor}[thm]{Corollary}
\newtheorem{con}{Question}
\newtheorem{conj}{Conjecture}
\newtheorem{definition}[thm]{Definition}
\newtheorem{remark}[thm]{Remark}
\theoremstyle{definition}
\newtheorem{ex}[thm]{Example}

\definecolor{RED}{rgb}{0.6,0,0}
\newcommand{\N}{\mathbb{N}}

\newcommand{\R}{\mathbb{R}}
\newcommand{\C}{\mathbb{C}}

\newcommand{\Sn}{\mathcal{S}_n}
\newcommand{\ff}{\text{f}}
\DeclareMathOperator{\sgn}{sgn}
\DeclareMathOperator{\Van}{Van}
\DeclareMathOperator{\CStab}{CStab}
\DeclareMathOperator{\RStab}{RStab}
\DeclareMathOperator{\Sym}{sym}
\DeclareMathOperator{\Kern}{Ker}
\DeclareMathOperator{\Hom}{Hom}
\DeclareMathOperator{\vol}{vol}

\numberwithin{thm}{section}
\newtheoremstyle{break}  
  {3pt}   
  {11pt}   
  {\normalfont}  
  {0pt}       
  {\scshape} 
  {}         
  {4pt}  
  {}          
\theoremstyle{break}

\begin{document}
\title{Symmetric nonnegative forms and sums of squares}

\author{Grigoriy Blekherman}
\address{School of Mathematics,
Georgia Tech,
686 Cherry Street,
Atlanta, GA 30332}

\author{Cordian  Riener}
\address{Department of Mathematics and Statistics, 
UiT -- The Arctic University of Norway, 9037 Troms\o}

\begin{abstract}

We study symmetric nonnegative forms and their relationship with symmetric sums of squares. For a fixed number of variables $n$ and degree $2d$, symmetric nonnegative forms and symmetric sums of squares form closed, convex cones in the vector space of $n$-variate symmetric forms of degree $2d$. Using representation theory of the symmetric group we characterize both cones in a uniform way.
Further, we investigate the asymptotic behavior when the degree $2d$ is fixed and the number of
variables $n$ grows.  Here, we  show that, in sharp contrast to the general
case, the difference between symmetric nonnegative forms and sums of squares does
not grow arbitrarily large for any fixed degree $2d$. 
We consider the case of symmetric quartic forms in more detail and give a complete characterization of quartic symmetric sums of squares. Furthermore, we show that in degree $4$ 
the cones of nonnegative symmetric forms and symmetric sums of squares approach the same limit, thus these two cones asymptotically become closer as the number of variables grows. 
We conjecture that this is true in
arbitrary degree $2d$.
\end{abstract}

\maketitle
\section{Introduction}

Throughout the paper let $\R[X_1,\ldots, X_n]$ denote the ring of polynomials in $n$ real
variables and $H_{n,k}$ the set of homogeneous polynomials (forms) of degree $k$ in $\R[X_1,\ldots, X_n]$. Certifying that a form
$f\in H_{n,2d}$ assumes only nonnegative values is one of the fundamental
questions of real algebra. One such possible certificate is a
decomposition of $f$ as a sum of squares, i.e., one finds forms  $p_1,\ldots
p_m\in H_{n,d}$ such that $f=p_1^2+\ldots+p_m^2.$  In 1888 Hilbert  \cite{Hilbert} gave a
beautiful proof showing that in general not all  nonnegative forms can be
written as a sum of squares. In fact, he showed that the sum of squares
property only characterizes nonnegativity in the cases of binary forms, of
quadratic forms, and of ternary quartics. In all other cases there exist forms
that are non-negative but do not allow a decomposition as a sum of squares.
Despite its elegance, Hilbert's proof was not constructive. A constructive approach to Hilbert's proof appeared in an article by Terpstra \cite{Ter} 
in 1939,  but the first explicit example was found by Motzkin in 1965 \cite{motz} and an explicit example based on Hilbert's method was constructed by Robinson in 1969 \cite{robin}. We refer the interested reader to \cite{scheiderer2009positivity,reznick2000some} for more background on this topic.

The sum of squares decomposition of nonnegative polynomials has been the
cornerstone of recent developments in polynomial optimization.
Following ideas of Lasserre and Parrilo, polynomial optimization
problems, i.e. the task of finding $f^{*}=\min f(x)$ for a polynomial
$f$,
can be relaxed and transferred into semidefinite optimization
problems.  If $f-f^{*}$ can be written as a sum of squares, these
semidefinite relaxations
are in fact exact. Hence a better understanding of the difference of
sums of squares and nonnegative polynomials is highly desirable.

We study the case of forms in $n$ variables of degree $2d$ that are symmetric, i.e., invariant under the
action of the symmetric group $\mathcal{S}_n$ that permutes the
variables. Let $\R[X_1,\ldots, X_n]^S$ denote the ring of symmetric polynomials and $H^S_{n,2d}$ denote the real vector space of symmetric
forms of degree $2d$  in $n$ variables. Let $\Sigma_{n,2d}^S$ denote
be the cone of forms in $H^S_{n,2d}$ that can be decomposed as sums of
squares and $\mathcal{P}^S_{n,2d}$ be the cone of non-negative
symmetric forms.
Choi and Lam \cite{choi} showed that the following symmetric form of degree~4 in $4$ variables is non-negative but cannot be written as a sum of squares:
\[ \sum X_i^2X_j^2+\sum X_i^2X_jX_k-4X_1X_2X_3X_4.\]

Thus one can conclude that  $\Sigma_{4,4}^S\neq\mathcal{P}^S_{4,4}$ and
therefore even in the case of symmetric polynomials the sum of
squares property  already fails to characterize nonnegativity in the first
case covered by Hilbert's classical result.  These results have been recently extended by Goel, Kuhlmann and Reznick \cite{gkr} into a full characterization of equality cases between $\Sigma_{n,2d}^S$ and $\mathcal{P}_{n,2d}^S$.  Unfortunately, there are no other interesting cases of equality beyond those covered by Hilbert's Theorem.

The case of even symmetric forms has also received some attention. Choi, Lam and Reznick \cite{CLR} fully described the cones of even symmetric sextics in any number of variables, and showed that under some normalization these cones have the same limit as the number of variables grows. Harris \cite{harris} showed that even symmetric ternary octics are non-negative, only if they are sums of squares, providing a new interesting case of equality between nonnegative polynomials and sums of squares. Goel, Kuhlmann and Reznick \cite{gkr2} showed that there are no other interesting cases of equality beyond Harris' and Hilbert's results for even symmetric forms.

Additionally to  the qualitative statement of Hilbert's characterization, a quantitative understanding of the gap between sums of squares and nonnegative forms
has been studied by several authors. In particular, in \cite{B1} the first author added to the work of Hilbert by showing
that the  gap between sum of squares and nonnegative forms of fixed degree grows
infinitely large with the number of variables if the degree is at
least $4$. This result has been recently been refined by Ergur to the multihomogenous case \cite{erg}. In this article we study the relationship between symmetric sums of squares and symmetric nonnegative forms. In particular, we are interested in the asymptotic behavior of the cones, which we can  realize for example  as  symmetric  mean inequalities naturally associated to a
symmetric polynomial. The study of such symmetric inequalities has a
long history  (see for example \cite{CGS}) and it is an interesting question to ask when one can use sum of squares certificates to verify such an inequality.   For instance,   Hurwitz \cite{hur} showed that
a sum of squares decomposition can be used to verify the arithmetic mean-geometric mean inequality. Recently, Frenkel and Horv\'ath \cite{frenkel} studied the connection of Minkowski's inequality to sums of squares. 
Our results imply that a positive fraction of such inequalities come from sums of squares symmetric polynomials. 
Furthermore, in degree $4$ we show that a family of symmetric power mean inequalities is valid for all
$n$ if and only if each member can be written as a sum of squares. We conjecture that this holds for all degrees.
\section{Overview and main results}
\subsection{Symmetric sums of squares}
Symmetric polynomials are  classical objects in algebra. In order to represent symmetric polynomials, we will make use of the power sum polynomials.
\begin{definition}\label{def:ps}
 For $i\in\N$  define $$P_{i}^{(n)}:=X_1^i+\ldots+X_n^i$$ to be the $i$-th power sum polynomial. We will also work with the  power means: $$p_i^{(n)}:=\frac{1}{n}{P_i}^{(n)}.$$

\end{definition}
 It is known (for example \cite[2.11]{macdo}) that $\R[X_1,\ldots, X_n]^S$ is freely generated by the algebraically independent polynomials $P_1^{(n)},\ldots,P_n^{(n)}$. Hence it follows that every symmetric polynomial $f\in \R[X_1,\ldots, X_n]^S$ of degree $2d\leq n$ can uniquely be
written as $$f=g(P^{(n)}_1,\ldots,P^{(n)}_{2d})$$ for some polynomial $g\in\R[z_1,\ldots,z_{2d}]$, with $\deg_w g=\deg f$, where  $\deg_w$ denotes the weighted degree corresponding to the weight $(1,\ldots,2d)$.
Recall that for a natural number $k$ a partition $\lambda$ of $k$
(written $\lambda\vdash k$)
is a sequence of weakly decreasing positive integers
$\lambda=(\lambda_1,\lambda_2,\ldots,\lambda_l)$ with
$\sum_{i=1}^l\lambda_i=k$.
For $n\geq k$ and to a partition $\lambda=(\lambda_1,\ldots,\lambda_l)\vdash k$ we associate 
polynomials \[P_\lambda^{(n)}:=P_{\lambda_1}^{(n)}\cdot
P_{\lambda_2}^{(n)}\cdots P_{\lambda_l}^{(n)}, \,\,\,\,\,\text{and} \,\,\,\,\, p_\lambda^{(n)}:=p_{\lambda_1}^{(n)}\cdot
p_{\lambda_2}^{(n)}\cdots p_{\lambda_l}^{(n)}.\]
It now follows  that for every $n\geq k$ the families  of polynomials  $\left\{P_\lambda\,|\, \lambda \vdash
k\right\}$ as well as $\left\{p_\lambda^{(n)}\,|\, \lambda \vdash
k\right\}$  form a basis of $H_{n,k}^S$. In particular,  if $n\geq k$ then the dimension of
$H_{n,k}^S$ is equal to $\pi(k)$, the number of partitions of $k$. Thus dimension of $H_{n,k}^S$ is constant for fixed $k$ and all sufficiently large $n$.

Using representation theory of the symmetric group, and in particular so-called higher Specht polynomials, we are able to give a uniform representation of the cone of symmetric sums of squares of fixed degree $2d$ in terms of matrix polynomials, with coefficients that are rational functions in $n$  (see Theorem \ref{thm:GP3}) and similarly a uniform representation of the sequence of dual cones in terms of linear matrix polynomials whose coefficients ``symmetrizations'' of sums of squares in $2d$ variables.  This gives us in particular a better understanding of the faces of $\Sigma^S_{n,2d}$ that are not faces of $\mathcal{P}^S_{n,2d}$.  We make these findings more concrete in the case of quartic symmetric  forms, where we completely characterize  the cone $\Sigma_{n,4}$ and its boundary. This in particular allows us to easily compute a family of symmetric sums of squares polynomials that are on the boundary of $\Sigma_{n,4}^S$ without having a real zero, thus certifying the difference of symmetric sums of squares and symmetric non-negative forms (see Theorem \ref{thm:1888}).

\subsection{Asymptotic behavior of sums of squares and nonnegative forms} 
Our characterization allows us to study   the asymptotic relationship between symmetric sums of squares and symmetric nonnegative forms of fixed degree in a growing number of variables. Even though vector spaces $H^{S}_{n,2d}$ have the same dimension $\pi(2d)$ for all $n\geq 2d$, there is no canonical way to identify vector spaces $H^{S}_{n,2d}$ for different $n$. In fact there are several natural ways to define transition maps identifying vector spaces of symmetric forms in different numbers of variables (see for example \cite{Transition}), and different transition maps will lead to different limits as $n$ goes to infinity. The system of vector spaces $H^S_{n,2d}$ together with transition maps will define a \textit{directed system} of vector spaces, and we can define the \textit{direct limit} $H^S_{\infty,2d}$ of vector space $H_{n,2d}^S$ \cite[Section 7.6]{advanced}.

\indent One way of defining  these transitions is by symmetrization:
\begin{definition}
For $f\in\R[X]$ we define the symmetrization of $f$ as $$\Sym_n(f):=\frac{1}{n!}\sum_{\sigma\in \mathcal{S}_n}\sigma(f).$$ 
 \end{definition}
The  composition of the natural inclusion $i_{n,n+1}: H_{n,2d} \rightarrow H_{n+1,2d}$ with $\Sym_{n+1}$ defines injective maps $\varphi_{n,n+1}:H_{n,2d}^S\rightarrow H_{n+1,2d}^S$. 

Therefore, we have the following.
 \begin{prop}
For $n,m\in\N$ with $n>m$  consider the maps $\varphi_{m,n}: H_{m,2d}^S\rightarrow H_{n,2d}^S$ defined  by
$$\varphi_{m,n}(p)=\Sym_n(p).$$
Then, the system of vector spaces  $H^S_{n,2d}$ together with the maps $\varphi_{m,n}$ defines a directed system and for $m\geq  2d$ the maps $\varphi_{m,n}$ are isomorphisms. \end{prop}

We consider the direct limit $H_{\infty,k}^{\varphi}$ of the directed system above. Since the maps $\varphi_{m,n}$ are isomorphisms with $m \geq 2d$, it follows that  $H_{\infty,k}^{\varphi}$ is also a real vector space of dimension $\pi(2d)$. Therefore we have natural isomorphisms $\varphi_n: H_{\infty,2d}^\varphi\rightarrow H_{n,2d}^S$ for $n \geq 2d$, which allow us to view the cones $\Sigma_{n,2d}^S$ and $\mathcal{P}_{n,2d}^S$ as subsets of $H_{\infty,2d}^\varphi$. Note that we have  $\varphi_{m,n}(\Sigma_{m,2d}^S)\subseteq \Sigma_{n,2d}^S$ and $\varphi_{m,n}(\mathcal{P}^S_{m,2d})\subseteq \mathcal{P}^S_{n,2d}$. It follows that with transition maps $\varphi_{m,n}$ the cones of sums of squares and the cones of nonnegative polynomials form nested increasing sequences in $H_{\infty,2d}^\varphi$. We define the following cones of nonnegative elements and sums of squares in $H_{\infty,k}^{\varphi}$:

$$\mathcal{P}_{\infty,2d}^{\varphi}:=\left\{\ff\in H_{\infty,2d}^{\varphi}\,:\,
\varphi_n(\ff)\in\mathcal{P}^S_{n,2d}~\text{ for all }~ n\geq 2d \right\},$$

$$\Sigma_{\infty,2d}^\varphi:=\left\{\ff\in H_{\infty,2d}^{\varphi}\,:\,
\varphi_n(\ff) \in \Sigma^S_{n,2d}~\text{
for all }~ n\geq 2d \right\}.$$
 
The following Theorem is immediate from the above discussion.
\begin{thm}\label{THM FD}
The cones $\mathcal{P}_{\infty, 2d}^{\varphi}$ and $\Sigma_{\infty,2d}^{\varphi}$ are full-dimensional convex cones in $H_{\infty,2d}^\varphi\simeq\R^{\pi(2d)}$.
\end{thm} 
Forms in fixed degree make up a vanishingly small portion of nonnegative forms as the number of variables grows \cite{B1}.  More precisely (non-symmetric), nonnegative forms and sums of squares in $n$ variables of degree $2d$ with average $1$ on the unit sphere form compact convex sets $\bar{P}_{n,2d}$ and $\bar{\Sigma}_{n,2d}$ of dimension $D=\binom{n+d-1}{d}-1$. It was shown in \cite{B1} that the ratio of volumes
$$\left(\frac{\vol \bar{\Sigma}_{n,2d}}{\vol \bar{P}_{n,2d}}\right)^{\frac{1}{D}}$$
converges to $0$ for all $2d\geq4$ as $n$ goes to infinity. The ratio of volumes is raised to the power $1/D$ to take into account the effects of large dimension on volumes as the volume of $(1+\varepsilon)\Sigma_{n,2d}$ is equal to $(1+\varepsilon)^D\vol \Sigma_{n,2d}$.

By contrast, the cones of symmetric nonnegative forms and sums of squares of fixed degree live in the vector space $H^S_{n,2d}$ which has fixed dimension $\pi(2d)$ for a sufficiently large number of variables $n$. Therefore, to prove that asymptotically symmetric sums of squares make up a nontrivial portion of symmetric nonnegative forms (with respect to some transition maps) it suffices to show that both limits are full-dimensional in $H_{\infty,2d}^\varphi\simeq\R^{\pi(2d)}$, which is done in Theorem \ref{THM FD}.

Besides the direct limit we also study \textit{symmetric power mean inequalities}. We can express a symmetric form $f$ in $H^S_{n,2d}$ in the power mean basis $p_{\lambda}^{(n)}$ with $\lambda \vdash 2d$:
$$f=\sum_{\lambda \vdash 2d} c_{\lambda}p_{\lambda}^{(n)}.$$
Using the power mean basis we can define transition maps $\rho_{m,n}$ by identifying
$$\sum_{\lambda \vdash 2d} c_{\lambda}p_{\lambda}^{(m)} \,\,\,\,\, \text{with}\,\,\,\,\, \sum_{\lambda \vdash 2d} c_{\lambda}p_{\lambda}^{(n)}.$$
As before the system of vector spaces  $H^S_{n,2d}$ together with the maps $\rho_{m,n}$ defines a directed system, and for $m\geq  2d$ the maps $\rho_{m,n}$ are isomorphisms. We consider the direct limit $H_{\infty,k}^{\rho}$. Since the maps $\rho_{m,n}$ are isomorphisms with $m \geq 2d$, it follows that  $H_{\infty,k}^{\rho}$ is again a real vector space of dimension $\pi(2d)$. The natural isomorphisms $\rho_n: H_{\infty,2d}^\rho\rightarrow H_{n,2d}^S$ for $n \geq 2d$, allow us to view the cones $\Sigma_{n,2d}^S$ and $\mathcal{P}_{n,2d}^S$ as subsets of $H_{\infty,2d}^\rho$. We will denote these images by $\Sigma_{n,2d}^\rho$ and $\mathcal{P}_{n,2d}^\rho$ and consider the limit cones:

\begin{definition}\label{def:rey}

$$\mathfrak{P}_{2d}:=\left\{\mathfrak{f} \in H^\rho_{\infty,2d}\,:\,
\rho_n(\mathfrak{f})\in\mathcal{P}^S_{n,2d}~\text{ for all }~ n\geq 2d \right\}$$
\noindent and
$$\mathfrak{S}_{2d}:=\left\{\mathfrak{f} \in H^\rho_{\infty,2d}\,:\,
\rho_n(\mathfrak{f})\in\Sigma^S_{n,2d}~\text{ for all }~ n\geq 2d \right\}.$$
\end{definition}

The  sequences $\mathcal{P}^{\rho}_{n,2d}$ and $\Sigma_{n,2d}^{\rho}$ are not nested in general. Let $x=(X_1,\dots,X_n)$ be a point in $\R^n$ and let $\tilde{x}$ be
the point in $\R^{k \cdot n}$ with each $X_i$ repeated $k$
times. Then, 
 $$p_i^{(k\cdot n)}(\tilde{x})=\frac{1}{k\cdot n}(k X_1^i+\ldots
+k X_n^i)=p_i^{(n)}(x).$$ It follows that $f^{(k\cdot 
n)}\in\mathcal{P}^{p}_{k\cdot  n,d}\Rightarrow
f^{(n)}\in\mathcal{P}^{p}_{n,d}$ and hence we get the following.
\begin{prop}\label{prop:inclusion}
Consider the cones $\mathcal{P}^{p}_{n,2d}$ as convex subsets of
$\R^{\pi(d)}$ using the coefficients $c_{\lambda}$ of $p_{\lambda}$.
Then for every $n\geq 2d$ and $k\in\N$ we have
$$\mathcal{P}^{\rho}_{k
\cdot n,2d}\subseteq\mathcal{P}^{\rho}_{n,2d}\subset H^{\rho}_{\infty,2d}\simeq\R^{\pi(2d)}.$$ \end{prop}
\begin{remark}\label{REM SOS}
We note that the same proof also yields that
$\Sigma^{\rho}_{k \cdot n,2d}\subseteq\Sigma^{\rho}_{n,2d}$.
\end{remark}
It is not directly clear from  Proposition \ref{prop:inclusion} that the sequences $\mathcal{P}^{\rho}_{n,2d}$ and $\Sigma^{\rho}_{n,2d}$ have limits, which we show separately:
\begin{thm}\label{THM LIMIT}
\begin{enumerate}
\item[$(a)$] The cones $\mathfrak{S}_{2d}$ and $\mathfrak{P}_{2d}$ are full dimensional cones.
\item[$(b)$] \[\mathfrak{P}_{2d}=\lim_{n\rightarrow \infty} \mathcal{P}^{\rho}_{n,2d} \quad \text{and} \quad \mathfrak{S}_{2d}=\lim_{n\rightarrow \infty} \Sigma^{\rho}_{n,2d}.\]
\end{enumerate}
\end{thm}

Although the cone of symmetric nonnegative quartics is  strictly bigger than the cone of symmetric quartic sums of squares for any number of variables $n\geq 4$, we show that in the limit the two cones coincide:

\begin{thm}\label{THM DEG4}
\[\mathfrak{P}_{4}=\mathfrak{S}_{4}.\]
\end{thm} 
In particular, this result applies in the situation of power mean inequalities studied in \cite{Rez1}, and hence it is possible to verify any such inequality using sums of squares.
We conjecture that this happens in arbitrary degree $2d$, i.e., we suggest the following.
 \begin{conj}\label{CONJ MAIN}
\[\mathfrak{P}_{2d}=\mathfrak{S}_{2d} \quad \text{for all} \quad d \in \mathbb{N}.\]
\end{conj}

\subsection{Structure of the article and guide for the reader}
\paragraph{This article is structured as follows:} 
We provide a characterization of symmetric non-negative forms and the limit cone in Section \ref{sec:psd}.  Section \ref{SEC Symsos} provides a detailed study of symmetric sums of squares.  To this end we present the general framework of how to use representation theory to study invariant sums of squares in Subsection \ref{subsec:SOS}. In Subsection \ref{sub-sym} we outline the basic notions of the representation theory of the symmetric group. These results  are then used in  Subsection 
\ref{sub:higherspecht} to represent the cone of symmetric sums of squares  (without restrictions on the degree) in terms of matrix polynomials in Theorems \ref{thm:GP2} and \ref{thm:GP2'}. The subsequent  Subsection \ref{sub:fixeddegree} then discusses how restricting degree  allows for a uniform description of the cones $\Sigma_{n,2d}$ in terms of the power mean bases $p_\lambda^{(n)}$ (Theorem \ref{thm:GP3}). The final subsection of Section  \ref{SEC Symsos} discusses some results on the dual cone with are needed in the sequel. The subsequent Section \ref{se:quartic} makes these results more concrete as we give a description of the cone of symmetric quartic sums of squares (Theorem \ref{thm:decom}). Furthermore, we describe the elements of the boundary of $\Sigma_{n,4}$ which are strictly positive in Theorem \ref{thm:boundary} and give an explicit example of such a polynomial for every $n\geq 4$ in Example \ref{ex:nice}. From this example it follows in particular that besides the cases where Hilbert showed the equality of sums of squares and non-negative forms there always exist symmetric positive definite forms which are not sums of squares (see Theorem \ref{thm:1888}).    In Section \ref{seq:limit} we explore the two notions of limits and prove Theorem \ref{THM LIMIT}. We also discuss the connection with the power mean inequalities. These power mean inequalities are then again studied in more detail In the final Section \ref{se:main},  where we show in particular that all valid power mean inequalities of degree 4 are sums of squares (Theorem  \ref{THM DEG4} ). 

The order of sections was chosen to present the more general statements  in Sections \ref{sec:psd}, \ref{SEC Symsos}, and \ref{seq:limit}  and then apply them in the quartic case in Sections \ref{se:quartic} and  \ref{se:main}.
Depending on reader's preferences they can also  begin  by reading Section \ref{se:quartic} first before actually diving into Section \ref{SEC Symsos} and similarly Section  \ref{se:main} before Section \ref{seq:limit}, 
while taking the necessary results from previous sections for granted. 

\section{Symmetric PSD forms}\label{sec:psd}

We begin by characterising the cone  $\mathfrak{P}_{2d}$. One key result needed to describe the non-negative 
 symmetric forms is the so-called  half degree principle (see
\cite{timofte-2003,Rie,Rie2}): For a natural number  $k\in \N$  we define
$A_k$ to be the set of all points in $\R^n$ with at most $k$ distinct
components, i.e.,
$$A_k:=\{x\in\R^n\,:\,|\{X_1,\ldots,X_n\}|\leq k\}.$$ The half degree
principle says that a symmetric form of degree $2d> 2$ is non-negative, if
and only if it is non-negative on $A_d$:
\begin{prop}[Half degree principle]\label{prop:degree}
Let $f\in H^S_{n,2d}$ and set $k:=\max\{2,d\}$. Then $f$ is
non-negative if and only if
$$f(y)\geq 0 \text{ for all } y\in A_{k}.$$
\end{prop}

\begin{remark}
By considering $f-\epsilon(X_1^2+\dots+X_n^2)^d$ for a sufficiently small $\epsilon >0$ we see that we can also replace \emph{non-negative} by \emph{positive} in the above
Theorem, thus characterizing strict positivity of symmetric forms. 
\end{remark}
A non increasing sequence   of $k$ natural numbers
$\vartheta:=(\vartheta_1,\ldots, \vartheta_k)$ such that
$\vartheta_1+\ldots+ \vartheta_k=n$ is called a $k$-partition of $n$
(written $\vartheta\vdash_k n$).  Given a symmetric form $f\in
H^S_{n,2d}$ and $\vartheta$ a $k$-partition of $n$ we define
$f^{\vartheta}\in\R[t_1,\ldots, t_k]$ via
$$f^{\vartheta}(t_1,\ldots,t_k):=f(\underbrace{t_1,\ldots,t_1}_{\vartheta_1},\underbrace{t_2,\ldots,t_2}_{\vartheta_2},\ldots,\underbrace{t_{k},\ldots,t_{k}}_{\vartheta_{k}}).$$

From now on assume that $2d>2$. Then the half-degree principle
implies that nonnegativity of $f=\sum_{\lambda\vdash
2d}c_{\lambda}p_{\lambda}$ is equivalent to nonnegativity of
$$f^{\vartheta}:=\sum\limits_{\lambda\vdash
2d}c_{\lambda}p_{\lambda}^{\vartheta}(t_1,\ldots t_k)$$ for all
$\vartheta\vdash_d n$, since the polynomials $f^{\vartheta}$ give the values of $f$ on all points with at most $d$ parts. We note that for all $i\in\N$ we have
$$p_i^{\vartheta}=\frac{1}{n}(\vartheta_1t_1^i+\vartheta_2t_2^i+\ldots+\vartheta_dt_d^i).$$
For a partition $\lambda=(\lambda_1,\dots,\lambda_l) \vdash 2d$ we
define a $2d$-variate form $\Phi_{\lambda}$ in the variables
$s_1,\dots,s_d$ and $t_1,\dots, t_d$ by
$$ \Phi_{\lambda}(s_1,\dots,s_d,t_1,\dots,t_d)=\prod_{i=1}^l(s_1t_1^{\lambda_i}+s_2t_2^{\lambda_i}+\ldots+s_dt_d^{\lambda_i})$$
\noindent and use it to associate to any form $f \in H^S_{n,2d}$,
$f=\sum_{\lambda\vdash 2d}c_{\lambda}p_{\lambda}$  the form
\begin{align*}\label{test}\Phi_f:=\sum_{\lambda\vdash
2d}c_{\lambda}\Phi_{{\lambda}}.
\end{align*}
\noindent Note that
$$\Phi_{\lambda}\left(\frac{\vartheta_1}{n},\dots\frac{\vartheta_d}{n},t_1,\dots
t_d\right)=p_{\lambda}^{\vartheta}(t_1,\ldots,t_d).$$ We define the
set  $$W_{n}=\left\{w=(w_1,\ldots w_d)\in \R^d \, | \, n\cdot w_i \in
\mathbb{N}\cup\{0\}, \text{ and } w_1+\ldots+w_d=1\right\}.$$

\noindent It follows from the arguments above that $f \in H^S_{n,2d}$
is non-negative if and only if the forms $\Phi_{f}(s,t)$ are
non-negative forms in $t$ for all $w \in W_n$. This is summarized in
the following corollary.
\begin{cor}\label{cor:degree}
Let $f=\sum_{\lambda\vdash 2d }c_{\lambda}p_{\lambda}$ be a form in
$H^S_{n,2d}$. Then $f$ is non-negative (positive) if and only if for
all $w\in W_n$
the $d$-variate forms $\Phi_f(w,t)$ are non-negative (positive).
\end{cor}

This result  enables us to characterize the elements of
$\mathfrak{P}_{2d}$. We expand the sets $W_n$ to the standard simplex
$\Delta$ in $\R^d$:
$$\Delta:=\left\{\alpha=(\alpha_1,\ldots,\alpha_d)\in[0,1]^d\,:\,
\alpha_1+\ldots+\alpha_d=1\right\}.$$

\noindent Then we have the following Theorem characterizing $\mathfrak{P}_{2d}$.

\begin{thm}\label{thm:charpos}
Let $\mathfrak{f}\in H^{\rho}_{\infty,2d}$ be the sequence defined by
$f^{(n)}=\sum_{\lambda\vdash 2d}c_\lambda p_\lambda^{(n)}$. Then
$\mathfrak{f}\in\mathfrak{P}_{2d}$ if and only if the $2d$-variate
polynomial
$\Phi_{f}(s,t)$ is non-negative on $\Delta\times \R^d$.
\end{thm}

\begin{proof}
Suppose that $\Phi_{f}(s,t)$ is non-negative on $\Delta\times \R^d$.
Let $f^{(n)}=\sum c_{\lambda}p^{(n)}_{\lambda}$. Since $W_{n} \subset
\Delta$ for all $n$ we see from Corollary \ref{cor:degree} that
$f^{(n)}$ is a non-negative form for all $n$ and thus $\mathfrak{f} \in
\mathfrak{P}_{2d}$. 

On the other hand, suppose there exists $\alpha_0 \in \Delta$ such
that $\Phi_{f}(\alpha_0,t)<0$ for some $t\in\R^{d}$. Then we can find
a rational point $\alpha \in \Delta$ with all positive coordinates and sufficiently close to $\alpha_0$
so that $\Phi_{f}(\alpha,t) < 0$.

Let $h$ be the least common multiple of the denominators of $\alpha$. Then we have $\alpha \in W_{ah}$ for all $a\in\N$. Choose $a$ such that $ah \geq 2d$. Then $f^{(ah)}$ is negative at the corresponding point and we have $\mathfrak{f} \notin
\mathfrak{P}_{2d}$.
\end{proof}

\section{Symmetric sums of squares}\label{SEC Symsos}

We now consider symmetric sums of squares. It was already observed in \cite{gatermann-parrilo-2004} that invariance under a group action
allows us to demand sum of squares decompositions which put strong restrictions on the underlying squares. First, we explain the general approach, which uses representation theory and can be used for other groups as well. Our presentation follows the ideas of \cite{gatermann-parrilo-2004} which we present in a slightly different way. The interested reader is advised to consult there for more details.

\subsection{Invariant Sums of Squares}\label{subsec:SOS}
Let $G$ be a finite group acting linearly on $\R^{n}$.
As $G$ acts linearly on $\R^n$ also the $\R-$vector space $\R[X]$ can be viewed as a $G$-module and by Maschke's theorem (the reader may consult for example \cite{serre-b77} for basics in linear representation theory) there exists a decomposition of the form
\begin{eqnarray}\label{eq:decomp}
\R[X] \ = \ V^{(1)} \oplus V^{(2)} \oplus \cdots \oplus V^{(h)} \,
\end{eqnarray}
with $V^{(j)} = W^{(j)}_1 \oplus \cdots \oplus W^{(j)}_{\eta_j}$ and $\nu_j := \dim W^{(j)}_i$.
Here, the $W^{(j)}_i$ are the \emph{irreducible components} and the $V^{(j)}$ are the \emph{isotypic components}, i.e., the direct sum of isomorphic irreducible components. The component with respect to the trivial irreducible
representation is the invariant ring $\R[X]^G$. The elements of the
other isotypic components are called \emph{semi-invariants}. It is classically known that each isotypic component is a finitely generated $\R[X]^{G}$-module (see \cite[Theorem 1.3]{stan}).
To any element $f\in H_{n,d}$ we can associate a \emph{symmetrization} by which we mean its image under the following linear map:
\begin{definition}
For a finite group $G$ the linear map $\mathcal{R}_G:\, H_{n,d}\rightarrow H_{n,d}^{G}$ which is defined by $$\mathcal{R}_G(f):=\frac{1}{|G|}\sum_{\sigma\in G}\sigma(f)$$ is called the Reynolds operator of $G$. In the case of $G=\mathcal{S}_n$ we say that $\mathcal{R}_{\mathcal{S}_n}(f)$ is a symmetrization of $f$ and we write $\Sym(f)$ in this case.
\end{definition}

For a set of polynomials $f_1,\ldots,f_l$ we will write $\sum\R\{f_1,\ldots,f_l\}^2$ to refer to the  sums of squares of elements in the linear span of the polynomials $f_1,\ldots,f_l$. It has already been observed by Gaterman and Parrilo \cite{gatermann-parrilo-2004} that invariant
sums of squares can be written as sums of squares of semi-invariants
using Schur's Lemma. However, a closer inspection of the situation allows in many  cases - as for example in the
case of $\mathcal{S}_n$ - a finer analysis of the decomposition into sums of squares.
Consider a set of forms $\{f_{1,1},\ldots,f_{1,\eta_1},f_{2,1},\ldots,f_{h,\eta_h}\}$ such that for 
  fixed $j$ the forms $f_{j,i}$  generate irreducible components
of $V^{(j)}$. Further assume that they
  are chosen in such a way, that for each $j$ and each pair $(l,k)$ there exists a
  $G$-isomorphism $\rho_{l,k}^{(j)}: V^{j}\rightarrow V^{j}$ which maps  $f_{j,l}$ to $f_{j,k}$. \
Now for every $j$ we consider the set $\{f_{j,1},\ldots f_{j,\eta_j}\}$ which contains only one polynomial per irreducible module. However, since every irreducible module is generated by the $G$-orbit of only one element, every such set uniquely describes the chosen decomposition. We call such a set a \emph{symmetry basis} and show that invariant sums of squares are in fact symmetrizations of sums of squares of a symmetry basis. The following theorem, which we state in a slightly more general setup highlights the use of a symmetry basis.

\begin{thm}\label{THM Decomp}
Let $G$ be a finite group and assume that all  real irreducible representations $V\subset H_{n,d}$ are also irreducible over their complexification. Let $p$ be a form of degree $2d$ which is invariant with
respect to $G$. If $p$ is a sum of squares, then $p$ can be written in the
form
$$p=\sum_{j}^{h} q_j,\text{ where each } q_j\in \sum\R\left\{f_{j,1},\ldots f_{j,\eta_j}\right\}^2.$$
\end{thm}
The main tool for the proof is Schur's Lemma, and we remark that a dual version of this Theorem can be found in \cite[Theorem 3.4]{RLJT} and \cite{riener2011symmetries}.
\begin{proof} 
Let  $p\in H_{n,2d}$ be a $G$-invariant sum of squares. Then there exists a 
symmetric positive semidefinite bilinear form $$B: H_{n,d}\times H_{n,d}\rightarrow \R$$ which is a Gram matrix for $p$, i.e.  for every $x\in\R^{n}$ we can write $p(x)=B(X^{d},X^{d})$, where $X^{d}$ stands for the 
$d$-th power of $x$ in the symmetric algebra of $\R^{n}$. Since $p$ is $G$-invariant, we have $p=\mathcal{R}_G(p)$ and by linearity we may assume that  $B$ is a $G$-invariant bilinear form. 
Now decompose $H_{n,2d}$ as in \eqref{eq:decomp} and consider the restriction of $B$ to $$B^{ij}:V^{(i)}\times V^{(j)}\rightarrow \R \text{ with } i\neq j.$$ For every $v\in V^{(i)}$  the quadratic form $B^{ij}$ defines a linear map $\phi_v:\, V^{(j)}\rightarrow \R$ via $\phi_v(w):=B^{ij}(v,w)$ and so $B^{ij}$ naturally  can be seen as an element of $\Hom^{G}({V^{(i)}}^{*},V^{(j)})$. Since real representations are self dual we have that ${V^{(i)}}^{*}$ and $V^{(j)}$ are not isomorphic and thus by Schur's Lemma   we find that $B^{ij}(v,w)=0$ for all $v\in V^{(i)}$ and $w\in V^{(j)}$. 
So the isotypic components are orthogonal with respect to $B$ and hence it suffices to look at $$B^{jj}: V^{(j)}\times V^{(j)}\rightarrow \R$$ individually. 
We have $V^{(j)}:=\bigoplus_{k=1}^{l} W^{(j)}_{k}$, where each $W^{(j)}_k$ is generated by a semi-invariant $f_{j,k}$, i.e. there is a basis $f_{j,k,1},\ldots,f_{j,k,\nu_j}$ for every $W^{(j)}_k$ such that the basis elements $f_{j,k,i}$ are taken from the orbit of $f_{j,k}$ under $G$. To again use Schur's Lemma we identify $B_j$ with its complexification $B_j^{\C}$, which is possible since we assumed that all representations are irreducible also over $\C$.
Consider a pair $W^{(j)}_{k_1}, W^{(j)}_{k_2}$, where we allow $k_1=k_2$. To apply  Schur's Lemma we relate the quadratic from  $B^{jj}$  to  a linear map 
$\psi^{(j)}_{k_1,k_2}~: W^{(j)}_{k_1}\rightarrow W^{(j)}_{k_2}$ defined on the generating set $f_{j,k_1,1},\ldots,f_{j,k_1,\nu_j}$ by $$\psi^{(j)}_{k_1,k_2}(f_{j,k_1,u}):=\sum_{v}  B^{jj}(f_{j,k_1,u},f_{j,k_2,v}) f_{j,k_2,v}.$$

Since we assumed that $W^{(j)}_{k}$ are absolutely irreducible we have by Schur's Lemma $$\dim(\Hom^G(W^{(j)}_{k_1},W^{(j)}_{k_2}))=1$$ and we can conclude that this map is unique up to scalar multiplication. Therefore it can be represented in the form $\psi^{(j)}_{k_1,k_2}=c_{k_1,k_2}\cdot\rho_{k_1,k_2}$, where $\rho_{k_1,k_2}$ is the $G$-isomorphism  with $\rho_{k_1,k_2}(f_{j,k_1})=f_{j,k_2}$ as above. It therefore follows that
$$B^{jj}(f_{j,k_1,u}, f_{j,k_2,v})=\delta_{u,v}c_{k_1,k_2},$$ where $\delta_{u,v}$ denotes the Kronecker Delta. By considering the matrix of $B$ with respect to the basis $f_{j,k,l}$ of $H_{n,d}$ we see that $p$ has the desired decomposition.


\end{proof}
\begin{remark}
The above statement also holds true in the situation where one looks at sums of squares of elements of an arbitrary $G$-closed submodule $T\subset\R[X]$.
\end{remark}

In some situations it is convenient to formulate the above Theorem \ref{THM Decomp} in terms of matrix polynomials, i.e. matrices with polynomial entries. Given two $k\times k$ symmetric matrices $A$ and $B$ define their inner product as $\langle A,B \rangle=\operatorname{trace}(AB).$ Define a block-diagonal symmetric matrix $A$ with $h$ blocks $A^{(1)},\dots,A^{(h)}$ with the entries of each block given by:
\begin{equation*}
A^{(j)}_{ik}=g_{ik}^{(j)}=\mathcal{R}_G(f_{j,i}\cdot f_{j,k}).
\end{equation*}
Then Theorem \ref{THM Decomp} is equivalent to the following statement:
\begin{cor}\label{COR Decomp}
With the conditions as in Theorem \ref{THM Decomp} let  $p\in \R[X]^G$. Then $p$ is a sum of squares of polynomials in $T$
 if and only if $p$ can be written as
\begin{equation*}
p=\langle A,B \rangle,
\end{equation*}
where $B$ is a positive semidefinite matrix with real entries.
\end{cor}

We now aim to apply Theorem \ref{THM Decomp} to a symmetric form $p \in H_{n,2d}^S$. In order to do this we need to identify an explicit representative in every irreducible $\Sn$-submodule of $H_{n,d}$. 
We first recall some useful facts from the representation theory of
$\mathcal{S}_n$. The irreducible representations in this case are the so-called \emph{Specht Modules}, which we will define in the following section. We refer to \cite{james-kerber-b81,sagan-2001} for
more details.
\subsection{Specht Modules as Polynomials}\label{sub-sym}
Let  $\lambda=(\lambda_1,\lambda_2,\ldots,\lambda_l)\vdash n$ be a partition of $n$. A
\emph{Young tableau} of shape $\lambda$ consists of $l$ rows, with
$\lambda_i$
entries in the $i$-th row.
Each entry is an element in $\{1, \ldots, n\}$, and each of these
numbers occurs
exactly once.. 
A \emph{standard Young tableau} is a Young tableau in which all rows and
columns are increasing.
An element $\sigma \in \mathcal{S}_n$ acts on a Young tableau by
replacing each entry by its image under $\sigma$.
Two Young tableaux $T_1$ and $T_2$ are called \emph{row-equivalent}
if the corresponding rows of the two tableaux contain the same numbers.
The classes of row-equivalent Young tableaux are called \emph{tabloids}, and
the equivalence class
of a tableau $T$ is denoted by $\{T\}$. The stabilizer of a
row-equivalence class is called the row-stabilizer denoted by $\RStab_T$. If $R_1,\ldots,R_l$ are the rows or a given Young tableau $T$ this group can be written as 
$$\RStab_T \ = \ \mathcal{S}_{R_1}\times \mathcal{S}_{R_2}\times
\cdots \times \mathcal{S}_{R_l}, $$
where $\mathcal{S}_{R_i}$ is the symmetric group on the elements of row $i$.
The action of $\mathcal{S}_n$ on the equivalence classes of row-equivalent Young tableaux gives rise to the \emph{permutation
module} $M^{\lambda}$
\emph{corresponding to} $\lambda$ which is the $\mathcal{S}_n$-module
defined by  $$M^\lambda=\R\left\{ \{T_1\}, \ldots ,\{T_s\}\right\},$$
where $\{T_1\}, \ldots, \{T_s\}$ is a complete list of
$\lambda$-tabloids and $\R\{ \{T_1\}, \ldots ,\{T_s\}\}$ denotes their
$\R$-linear span. 

Let $T$ be a Young tableau for $\lambda\vdash n$, and let $C_i$ be the
entries in the $i$-th column of $t$.
The group $$\CStab_T \ = \ \mathcal{S}_{C_1}\times \mathcal{S}_{C_2}\times
\cdots \times \mathcal{S}_{C_\nu},$$ 
where $\mathcal{S}_{C_i}$ is the symmetric group elements of columns $i$,
is called the \emph{column stabilizer} of $T$. The irreducible
representations of the symmetric group $\mathcal{S}_n$
are in 1-1-correspondence
with the partitions of $n$, and they are given by the Specht modules, as
explained below. For $\lambda\vdash n$, the
\emph{polytabloid associated with} $T$ is
defined by
$$e_T \ = \ \sum_{\sigma\in \CStab_t} \sgn(\sigma)\sigma\{t\} \, .$$
Then for a partition $\lambda\vdash n$, the \emph{Specht module}
$S^{\lambda}$
is the submodule of the permutation module $M^\lambda$ spanned by the
polytabloids $e_T$. The dimension of $S^{\lambda}$ is given by the
number of standard
Young tableaux for $\lambda \vdash n$, which we will denote by
$s_\lambda$.  

A classical construction of Specht realizes Specht
modules as submodules of the polynomial ring (see \cite{specht-1933}): For $\lambda\vdash n$ let $T_{\lambda}$ be a  standard Young tableau of shape $\lambda$ and
$\mathcal{C}_1,\ldots,\mathcal{C}_{\nu}$ be the columns of
$T_\lambda$. To $T_\lambda$ we associate the monomial
$X^{t_{\lambda}}:=\prod_{i=1}^{n}X_i^{m(i)-1}$, where $m(i)$ is the
index of the row of $T_{\lambda}$ containing $i$.
Note that for any $\lambda$-tabloid $\{T_{\lambda}\}$ the monomial
$X^{T_{\lambda}}$ is well defined,
and the mapping $\{T_{\lambda}\} \mapsto X^{T_{\lambda}}$ is an
$\mathcal{S}_n$-isomorphism.
For any column $\mathcal{C}_i$ of $T_\lambda$ we denote by
$\mathcal{C}_i(j)$ the  element in the  $j$-th row and we associate to it a
Vandermonde determinant:
$$\Van_{\mathcal{C}_{i}} \ := \ \det
\left(
\begin{array}{ccc}
X_{ \mathcal{C}_i(1)}^0& \ldots  &X_{\mathcal{C}_i(k)}^0   \\
 \vdots&  \ddots &\vdots   \\
X_{ \mathcal{C}_i(1)}^{k-1}& \ldots  &X_{\mathcal{C}_i(k)}^{k-1}
\end{array}
\right) \ = \ \prod_{j<l}(X_{\mathcal{C}_i(j)}-X_{\mathcal{C}_i(l)}).$$
\noindent The \emph{Specht polynomial} $sp_{T_{\lambda}}$ associated
to $T_\lambda$ is defined as
\[
  sp_{T_{{\lambda}}} \ := \ \prod_{i=1}^{\nu} \Van_{\mathcal{C}_{i}}
  \ = \
 \sum_{\sigma\in \CStab_{T_{\lambda}}}\sgn(\sigma)\sigma(X^{T_{\lambda}}) \, ,
\]
where $\CStab_{T_{\lambda}}$ is the column stabilizer of $T_\lambda$.

\noindent By the $\mathcal{S}_n$-isomorphism $\{T_{\lambda}\} \mapsto
X^{t_{\lambda}}$,
$\mathcal{S}_n$ acts on $sp_{T_{{\lambda}}}$ in the same way as on
the polytabloid $e_{T_{\lambda}}$.
If $T_{\lambda,1},\ldots,T_{\lambda,k}$ denote all standard Young
tableaux associated to $\lambda$, then the set of polynomials
$sp_{T_{\lambda,1}},\ldots,s_{T_{\lambda,k}}$
are called the \emph{Specht polynomials} associated to $\lambda$. We then have the following Proposition \cite{specht-1933}:
\begin{prop}\label{pro:Specht}
The Specht polynomials $sp_{T_{\lambda,1}},\ldots,s_{T_{\lambda,k}}$ span an
$\mathcal{S}_n$-submodule of $\R[X]$ which is isomorphic to the Specht
module $S^{\lambda}$.
\end{prop}

The Specht polynomials identify a submodule of $\R[X]$ isomorphic to $\mathcal{S}^{\lambda}$. In order to get a decomposition of the entire ring $\R[X]$ we will use a generalization of this construction which is described in the next section.

\subsection{Higher Specht polynomials and the decomposition of $\R[X]$}\label{sub:higherspecht}
In what follows we will need to understand the decomposition of the polynomial ring $\R[X]$ and $\mathcal{S}_n$-module $H_{n,d}$ in terms of $\mathcal{S}_n$-irreducible representations. Notice that such a decomposition is not unique. 
It is classically known that the ring $\R[X]$ is a free module of dimension $n!$ over the ring of symmetric polynomials. Similarly, every isotypic component is a free $\R[X]^{\Sn}$-module. Therefore, one general strategy in order to get a symmetry basis of $\R[X]$ consists in building a free module basis for $\R[X]$ over $\R[X]^{\Sn}$ which additionally is symmetry adapted, i.e., which respects a decomposition into irreducible $\Sn$-modules.  One such construction, which generalizes Specht's original construction presented above is due to  Ariki, Terasoma, and Yamada \cite{ATY}.

\begin{definition}
Let $n\in\N$.
\begin{enumerate}
\item A finite sequence $w=(w_1,\ldots,w_n)$ of non-negative integers  is called a word of length $n$. A word $w$ of length $n$ is called a permutation if the set of non-negative integers forming a word of length $n$  is $\{1,\ldots,n\}$.
\item Given a word $w$ and a permutation $u$ we define the monomial associated to the pair as  $X_u^{w}:=X_{u_1}^{w_1}\cdots X_{u_n}^{w_n}$.
\item Given a permutation $w$. We associate to $w$ is index  denoted by $i(w)$, by constructing the following word of length $n$. The word $i(w)$ contains 0 exactly at the same position where $1$ occurs in $w$ and  the other entries we defined recursively with the following rule: Suppose that  the entry in $i(w)$ at a given position is $c$ and that $k$ occurs in $w$ at the same position then $i(w)$ should be also $c$ if it lies to the right of $k$ and it should be $c+1$ is it lies to the left of $k$.
\item For $\lambda\vdash n$ and $T$ be a standard Young tableau of shape $\lambda$ we define the \emph{word} of $T$ - denoted by $w(T)$ - by collecting the entries of $T$ from the bottom to the top in consecutive columns starting from the left. 
\item For a pair  $(T,V)$ of standard $\lambda$-tableaux we define the monomial associated to this pair as $X_{w(V)}^{i(w(T))}.$
\end{enumerate}
\end{definition}

%

\begin{ex}
Consider the tableau $$\Yvcentermath1 T=\young(124,35).$$ The resulting word is given by $$w(T)=31524,$$
with  $$i(w(T))=10001.$$

Taking $$\Yvcentermath1 V= \young(135,24)$$ we obtain $X_{w(V)}^{i(w(T))}=X_1^0X_2^1X_3^0X_4^2X_5^1$. 
\end{ex} 

\begin{definition}
Let $\lambda\vdash n$ and $T$ be a $\lambda$-tableau. Then the \emph{Young symmetrizer} associated to $T$ is the element in the group algebra $\R[\Sn]$ defined to be 
$$\varepsilon_T=\sum_{\sigma\in \RStab_T}\sum_{\tau\in \CStab_T} \sgn( \tau) \tau\sigma.$$
Now let $T$ be a standard Young tableau, and define the \emph{higher Specht polynomial} associated with the pair $(T,V)$ to be 
$$F_V^T(X_1,\ldots,X_n):=\varepsilon_V(X_{w(V)}^{i(w(T))}).$$

For $\lambda\vdash n$ we will denote by $$\mathcal{F}_\lambda=\{F_V^{T}, \text{ where }T,V \text{ run over all  standard }\lambda\text{ tableaux }\}$$ the set of all standard higher Specht polynomials corresponding to $\lambda$  and by $$\mathcal{F}=\bigcup_{\lambda\vdash n} \mathcal{F}_\lambda$$ the set of all standard higher Specht polynomials. 
\end{definition}
\begin{remark}
Let $s_{\lambda}$ denote the number of standard Young tableaux of shape $\lambda$.
It follows from the so-called Robinson-Schensted correspondence (see \cite{sagan-2001}) that 
$$\sum_{\lambda\vdash n} s_\lambda^2=n!.$$ Therefore the cardinality of $\mathcal{F}$ is exactly $n!$.
\end{remark}
The importance of the higher Specht polynomials now is summarized in the following Theorem which can be found in  \cite[Theorem 1]{ATY}.
\begin{thm}\label{thm:higher}
The following holds for the set of higher Specht polynomials.
\begin{enumerate}
\item The set $\mathcal{F}$ is a free basis of the ring $\R[X]$ over the invariant ring $\R[X]^{\Sn}$. 
\item For any $\lambda\vdash n$ and standard $\lambda$-tableau $T$, the space spanned by the polynomials in $$\mathcal{F}^T_\lambda:=\{F^T_V, \text{ where } V \text{ runs over all standard } \lambda\text{-tableaux} \}$$ is an irreducible $\mathcal{S}_n$-module isomorphic to the Specht module $S^{\lambda}$. 
\end{enumerate}
\end{thm}
For every $\lambda\vdash n$ we denote by $V_0^{\lambda}$ the standard $\lambda$ tableau with entries $\{1,\ldots,\lambda_1\}$ in the first row $\{\lambda_1+1,\ldots,\lambda_2\}$ in the second row and so on. 
Consider the set $$\mathcal{Q}_\lambda:=\{F^T_{V_0^{\lambda}}, \text{ where } T \text{ runs over all standard } \lambda\text{-tableaux} \},$$ which is of cardinality $s_\lambda$. The set $\mathcal{Q}_\lambda$ is a symmetry basis  of the vector space spanned by  $\mathcal{F}$. 
Using these polynomials we define  $s_\lambda\times s_\lambda$ matrix polynomials $Q^{\lambda}$ by:
\begin{equation}\label{eqn:Q}Q^\lambda({T,T'}):=\Sym(F^T_{V_0^{\lambda}}\cdot F^{T'}_{V_0^{\lambda}}),\end{equation}
where $T,T'$ run over all standard $\lambda$-tableaux. 
Since by $(1)$ in Theorem \ref{thm:higher} we know that every polynomial $h\in\R[X]$ can be uniquely written as a linear combination of elements in $\mathcal{F}$ with coefficients in $\R[X]^{\Sn}$, the following theorem can be thought of as a generalization of Corollary \ref{COR Decomp} to sums of squares from an $\mathcal{S}_n$-module with coefficients in an $\mathcal{S}_n$-invariant ring (see also \cite[Theorem 6.2] {gatermann-parrilo-2004}):

\begin{thm}\label{thm:GP2}
Let $p\in\R[X]^{\Sn}$ be a symmetric polynomial. Then $p$ is a sum of squares if and only if it can be written in the form
$$p=\sum_{\lambda\vdash n} \langle B^\lambda,Q^\lambda\rangle,$$
where $Q^\lambda$ is defined in \eqref{eqn:Q} and each $B^\lambda\in\R[X]^{s_\lambda \times s_\lambda}$ is a sum of symmetric squares matrix  polynomial, i.e. $$B^\lambda(x)=L^{t}(x)L(x)$$ for some matrix polynomial $L(x)$ whose entries are symmetric polynomials.\end{thm}

Each entry of the matrix $Q^{\lambda}$ is a symmetric polynomial and thus  can be represented as a polynomial in any set of generators of the ring of symmetric polynomials. We will use the the power means $p_1,\ldots,p_n$ to phrase the next theorem. However, any other choice works similarly. With this choice of basis it follows that there exists a matrix polynomial $\tilde{Q}^{\lambda}(z_1,\ldots,z_n)$ in $n$ variables $z_1,\ldots,z_n$
such that
\begin{equation}\label{eqn:Qt} \tilde{Q}^{\lambda}(p_1(x),\ldots,p_n(x))=Q^{\lambda}(x).\end{equation}

With this notation  can restate Theorem \ref{thm:GP2} in the following way:
\begin{thm}\label{thm:GP2'}
Let $f\in\R[X]^{\Sn}$ be a symmetric polynomial and $g\in\R[z_1,\ldots,z_n]$ such that $f=g(p_1,\ldots,p_n)$. 
Then $f$ is a sum of squares if and only if $g$ can be written in the form
$$g=\sum_{\lambda\vdash n} \langle B^\lambda,\tilde{Q}^\lambda\rangle,$$
where $\tilde{Q}^\lambda$ is defined in \eqref{eqn:Qt} and each $B^\lambda\in\R[z]^{s_\lambda \times s_\lambda}$ is a sum of squares matrix  polynomial, i.e. $B^\lambda:=L(z)^{t}L(z)$ for some matrix polynomial $L$. 
\end{thm}

While Theorems \ref{thm:GP2} and \ref{thm:GP2'} give a characterization of symmetric sums of squares in a given number of variables, we need to understand the behavior of the $\mathcal{S}_n$-module $H_{n,d}$ for polynomials of a fixed degree $d$ in a growing number of variables $n$.  This will be done in the next section.

%

\subsection{The cone $\Sigma_{n,2d}^S$}\label{sub:fixeddegree}
A symmetric sum of squares $f \in \Sigma^S_{n,2d}$ has to be a sum of squares from $H_{n,d}$. Therefore we now consider restricting the degree of the squares in the underlying sum of squares representation. With a little abuse of notation we denote by $\mathcal{F}_{n,d}$ be the vector space spanned by higher Specht polynomials for the group $\mathcal{S}_n$  of degree at most $d$.  Further, for a partition $\lambda \vdash n$ let $\mathcal{F}_{\lambda,d}$ denote the span of the higher Specht polynomials of degree at most $d$ corresponding to the Specht module $\mathcal{S}^{\lambda}$, i.e., $\mathcal{F}_{\lambda,d}$ is exactly  the isotypic component of $\mathcal{F}_{n,d}$ corresponding to  $\mathcal{S}^{\lambda}$.
In order to describe this isotypic component combinatorially,  recall that the degree of the higher Specht polynomial $F_T^S$ is given by the charge $c(S)$ of $S$. Thus, it follows from the above construction that $$\mathcal{F}_{\lambda,d}=\operatorname{span} \{F_T^{S}: S,T \text{ are standard }\lambda\text{-tableaux  and } c(S)\leq d\}.$$

We now show that sums of squares of degree $2d$ in $n$ variables can be constructed by symmetrizing sums of squares in $2d$ variables. So we first consider the case $n=2d$. Let $$\mathcal{F}_{2d,d}=\bigoplus_{\lambda \vdash 2d} m_{\lambda}S^{\lambda},$$
be the decomposition of $\mathcal{F}_{2d,d}$ as an $\mathcal{S}_{2d}$-module. The following Proposition gives the multiplicities of the different $\mathcal{S}_n$ modules appearing in the vector space of homogeneous polynomials of degree $d$.
\begin{prop}
The multiplicities $m_{\lambda}$ of the $\mathcal{S}_n$-modules $\mathcal{S}^\lambda$ which appear in an isotypic decomposition  $H_{n,d}$ coincide with  the number of standard $\lambda$-tableaux $S$ with the charge of $S$ at most $d$: $c(S)\leq d$.
\end{prop}

For a partition $\lambda \vdash 2d$ and $n \geq 2d$ define a new partition $\lambda^{(n)} \vdash n$ by simply increasing the first part of $\lambda$ by $n-2d$: $\lambda^{(n)}_1=\lambda_1+n-2d$ and $\lambda^{(n)}_i=\lambda_i$ for $i \geq 2$.  Then the decomposition  Theorem  \ref{thm:higher}   in combination with  \cite[Theorem 4.7.]{RLJT} yields that
$$\mathcal{F}_{n,d}=\bigoplus_{\lambda \vdash 2d} m_{\lambda}\mathcal{S}^{\lambda^{(n)}}.$$

For every $\lambda\vdash 2d$ we choose $m_\lambda$-many higher Specht polynomials 
$\{q_1^{\lambda},\ldots,q_{m_\lambda}^{\lambda}\}$ that form a symmetry basis of the $\lambda$-isotypic component of $\mathcal{F}_{2d,d}$.

 Let $q_\lambda=(q_1^{\lambda},\ldots,q_{m_\lambda}^{\lambda})$ be a vector with entries $q_i^{\lambda}$. As before we construct a matrix $Q^{\lambda}_{2d}$ by:
$$Q^{\lambda}_{2d}=\Sym_{2d}(q_{\lambda}^t q_\lambda) \qquad Q_{2d}^\lambda({i,j})=\Sym_{2d}(q^{\lambda}_{i}\cdot q^{\lambda}_{j}).$$ Further, we define a matrix $Q_n^{\lambda}$ by $$Q_{n}=\Sym_{n}(q_{\lambda}^t q_\lambda) \qquad Q_n^{\lambda}(i,j)=\Sym_n q_i^{\lambda}q_j^{\lambda}.$$
By construction we have the following:
\begin{prop}\label{COR Sym}
The matrix $Q_n^{\lambda}$ is the $\mathcal{S}_n$-symmetrization of the matrix $Q_{2d}^{\lambda}$:
$$Q_n^{\lambda}=\Sym_n Q_{2d}^{\lambda}.$$
\end{prop}

We now give a parametric description of the family of cones $\Sigma^S_{n,2d}$. Note again, that this statement is given in terms of a particular basis, but similarly can be stated with any set of generators.
\begin{thm}\label{thm:GP3}
Let $f:=\sum_{\lambda\vdash 2d} c_\lambda p_\lambda^{(n)}\in H_{n,2d}^S$. Then $f$ is a sum of squares if and only if it can be written in the form
$$f=\sum_{\lambda\vdash 2d} \langle B^\lambda,Q_n^\lambda\rangle,$$
where each $B^\lambda\in\R[p_1^{(n)},\ldots,p_{d}^{(n)}]^{m_\lambda \times m_\lambda}$ is a sum squares matrix of power sum  polynomials, i.e. $$B^\lambda=L_{\lambda}^{t}L_{\lambda}$$ for some matrix polynomial $L_{\lambda}(p_1^{(n)},\ldots,p_{d}^{(n)})$ whose entries are weighted homogeneous forms. 

Additionally, we have for every column $k$ of $L_{\lambda}$
$$\deg_w Q_n^{\lambda}(i,k)+2\deg_w L_{\lambda}(k,i)=2d,$$ or equivalently every entry $B^{\lambda}(i,j)$ of $B^{\lambda}$ is a weighted homogeneous form such that,
$$\deg_w Q_n^{\lambda}(i,j)+\deg_w B^{\lambda}(i,j)=2d.$$ 

\end{thm}
\begin{proof}
In order to apply Theorem \ref{thm:GP2} to our fixed degree situation we have to show that the forms $\{q_1^{\lambda},\ldots,q_{m_\lambda}^{\lambda}\}$ when viewed as functions in $n$ variables also form a symmetry basis of the $\lambda^{(n)}$-isotypic component of $\mathcal{F}_{n,d}$ for all $n \geq 2d$. Indeed consider a standard Young tableaux $t_{\lambda}$ of shape $\lambda$ and construct a standard Young tableau $t_{\lambda^{(n)}}$ of shape $\lambda^{(n)}$ by adding numbers $2d+1,\dots,n$ as rightmost entries of the top row of $t_{\lambda^{(n)}}$, while keeping the rest of the filling of $t_{\lambda^{(n)}}$ the same as for $t_{\lambda}$. It follows by construction of the Specht polynomials that $$sp_{t_{\lambda}}=sp_{t_{\lambda^{(n)}}}.$$
We may assume, that the $q_k^{(\lambda)}$ were chosen such that they map to $sp_{t_{\lambda}}$ by an $\mathcal{S}_{2d}$-isomorphism. We observe that $sp_{t_{\lambda}}$ (and therefore $sp_{t_\lambda^{(n)}}$) and $q_k^{\lambda}$ do not involve any of the variables $X_{j}$, $j>2d$. Therefore both are stabilized by $\mathcal{S}_{n-2d}$ (operating on the last $n-2d$ variables), and further the action on the first $2d$ variables is exactly the same. Thus there is an $\mathcal{S}_{n}$-isomorphism mapping $q_k^{\lambda}$ to $sp_{t_\lambda^{(n)}}$ and the $\mathcal{S}_{n}$-modules generated by the two polynomials are isomorphic. Therefore it follows that $q_k^{(\lambda)}$ also form a symmetry basis of the $\lambda^{(n)}$ isotypic component of $\mathcal{F}_{n,d}$. 
\end{proof}

\begin{remark}
We remark that the sum of squares decomposition of $f=\sum_{\lambda \vdash 2d} \langle B^{\lambda},Q^{\lambda}_n \rangle$, with $B^{\lambda}=L_\lambda^tL_{\lambda}$ can be read off as follows:
\begin{equation}\label{EQN decomp}f=\sum_{\lambda \vdash 2d} \Sym_n q_{\lambda}^t B^{\lambda}q_{\lambda} =\sum_{\lambda \vdash 2d} \Sym_n \left( L_{\lambda}q_{\lambda} \right)^tL_{\lambda}q_{\lambda} .\end{equation}
In particular, if for a fixed $\lambda\vdash n$ and  for every $1\leq i\leq m_\lambda$ we denote   $\delta_i:=d-\deg q_i^{\lambda}$, then the set  of polynomials

\begin{equation}\label{eq:basis}
 \bigcup_{i=1}^{m_\lambda}\bigcup_{\nu\vdash \delta_i}\left\{ q_i^{\lambda}\cdot p_\nu\right\}
 \end{equation}
 
is a symmetry basis of the isotypic component   $H_{n,d}$ corresponding to $\lambda$.
\end{remark}

\subsection{The dual cone of symmetric sums of squares}\label{sub:dual}

Recall, that for a convex cone $K\subset\R^n$ the dual
cone $K^*$ is defined as
$$K^{*}:=\{\l\in\text{Hom}(\R^n,\R)\,:\,\ell(x)\geq 0\, \quad \text{ for all } \quad
x\in K\}.$$ Our analysis of the dual cone $(\Sigma_{n,2d}^S)^*$ proceeds similarly
to the analysis of the dual cone in the non-symmetric situation given
in \cite{B2, blekherman2017extreme}. 

Let $S_{n,d}$ be the vector space of real quadratic forms on
$H_{n,d}$. Let  $S_{n,d}^+$ be the cone of positive semidefinite quadratic forms in  $S_{n,d}$. An element
$\mathcal{Q}\in S_{n,d}$ is said to be $\mathcal{S}_n$-invariant if
$\mathcal{Q}(f)=\mathcal{Q}(\sigma(f))$ for all
$\sigma\in \mathcal{S}_n$, $f \in H_{n,d}$. We will denote by
$\bar{S}_{n,d}$ the space of
$\mathcal{S}_n$-invariant quadratic forms on $H_{n,d}$. Further we can
identify a linear functional $\l\in (H_{n,2d}^{S})^*$  with a
quadratic form $\mathcal{Q}_{\l}$
defined by $$\mathcal{Q}_\ell(f) = \ell(\Sym(f^2)).$$

Let  $\bar{S}_{n,d}^+$ be the cone of positive semidefinite forms in
$\bar{S}_{n,d},$ i.e.,
$$\bar{S}_{n,d}^+:=\{\mathcal{Q}\in \bar{S}_{n,d}\,:\, \mathcal{Q}(f)\geq 0\text{ for all } f\in
H_{n,d}\}.$$

The following Lemma is straightforward, but very important, as it allows
us to identify the elements of dual cone  $\l\in(\Sigma_{n,2d}^S)^*$
with quadratic forms $\mathcal{Q}_{\ell}$ in $\bar{S}_{n,d}^+$.

\begin{lemma}\label{le:psd}
A linear functional $\ell \in(H_{n,2d}^S)^*$ belongs to the dual cone
$(\Sigma_{n,2d}^S)^*$ if and only if the quadratic form $\mathcal{Q}_{\ell}$ is
positive semidefinite.
\end{lemma}

Since  for $\ell \in(H_{n,2d}^S)^*$ we have $\mathcal{Q}_{\ell}\in\bar{S}_{n,d}$ Schur's Lemma again applies and we can use the symmetry basis constructed above to simplify the condition that $\mathcal{Q}_{\ell}$ is
positive semidefinite. In order to arrive at a dual statement of Corollary \ref{thm:GP3} we construct the following matrices:

\begin{definition}\label{def:M}
For every partition  $\lambda\vdash 2d$ consider the block-matrix $M_{n,\lambda}$ defined by
$$M_{n,\lambda}^{(i,j)}(\alpha,\beta):=\ell(p_{\alpha}\cdot p_{\beta}\cdot Q^{\lambda}_n(i,j)),$$
where in each block $i,j$ the indices $(\alpha,\beta)$ run through all pairs of weakly degreasing sequences $\alpha=(\alpha_1,\ldots,\alpha_a)$ and $\beta=(\beta_1,\ldots,\beta_b)$  such that $2d-\deg_w Q^{\lambda}_n(i,j)=\alpha_1+\ldots+\alpha_a+\beta_1+\ldots+\beta_b$.

\end{definition}
With this notation the following Lemma is just the dual version of Corollary \ref{thm:GP3} and is established by expressing Lemma  \ref{le:psd} in the basis given in \eqref{eq:basis}:
\begin{lemma}\label{thm:GP4}
Let $\ell\in H_{n,2d}^{*}$ be a linear functional. Then $\ell\in (\Sigma_{n,2d})^{*}$ if and only if for all $\lambda\vdash 2d$  the above matrices $M_{n,\lambda}$ are positive semidefinite.
\end{lemma}

\indent In order to examine the kernels of quadratic forms we use the
following construction. Let $W\subset H_{n,d}$ be any linear subspace.
We define $W^{<2>}$ to be the symmetrization of the degree $2d$ part of
the ideal generated by $W$:
$$W^{<2>}:=\left\{ h\in H_{n,2d}^S\,:\;h=\Sym\left(\sum f_i
g_i\right)\;\text{ with } f_i \in W \text{ and  }g_i \in
H_{n,d}\right\}.$$

In Lemma \ref{le:psd} we identified the
dual cone  $(\Sigma_{n,2d}^{S})^*$ with a
linear section of the cone of positive semidefinite quadratic forms  $S_{n,d}^+$ with the
subspace $\bar{S}_{n,d}$ of symmetric quadratic forms. By a slight
abuse of terminology we think of positive semidefinite forms
$Q_{\ell}$ as elements of the dual cone $(\Sigma_{n,2d})^*$. The following important Proposition is a straightforward adaptation of the equivalent result in the non-symmetric case \cite[Proposition 2.1]{blekherman2017extreme}:

\begin{prop}\label{COR HYPER}
Let $\ell \in (\Sigma^{S}_{n,2d})^*$ be a linear functional non-negative on squares and let $W_{\ell}\subset H_{n,d}$ be the kernel of the quadratic form $\mathcal{Q}_{\ell}$. The linear functional $\ell$ spans an extreme ray of 
$(\Sigma_{n,2d}^S)^*$ if and only if $W_{\ell}^{<2>}$ is a hyperplane in $H_{n,2d}^S$. Equivalently, the kernel of $\mathcal{Q}_\ell$ is maximal, i.e. if $\ker \mathcal{Q}_\ell\subseteq \ker \mathcal{Q}_m$ for some $m \in H_{n,2d}^*$ then $m=\lambda\ell$ for some $\lambda \in\R$.
\end{prop}

\indent The dual correspondence yields that any facet $F$ of a cone $K$,
i.e. any maximal face of $K$, is given by an extreme ray of the dual cone $K^*$.
More precisely, for any maximal face $F$ of $K$ there exists an
extreme ray of $K^*$ spanned by a linear functional $\ell \in K^*$
such that $$F=\{x\in K\,:\, \text{such that } \ell(x)=0\}.$$

We now aim to characterize the extreme
rays of $(\Sigma_{n,2d}^{S})^*$ which are not extreme rays of the cone $(\mathcal{P}_{n,2d}^{S})^*$. For $v \in \R^n$ define a linear functional 
$$\ell_v:H^S_{n,2d} \rightarrow R\quad \ell_v(f)=f(v).$$
We say that the linear functional $\ell_v$ corresponds to point evaluation at $v$. It is easy to show with the same proof as in the
non-symmetric case that the extreme rays of the cone $(\mathcal{P}_{n,2d}^{S})^*$ are precisely the point evaluations $\ell_v$ (see \cite[Chapter 4]{BPT} for the non-symmetric case). Therefore we need to identify extreme rays of $(\Sigma_{n,2d}^{S})^*$ which are not point evaluations. We now examine the case of degree $4$ in detail, and give an explicit construction of an element of $(\Sigma_{n,4}^{S})^*$, which does not belong to $(\mathcal{P}_{n,2d}^{S})^*$.

\section{Symmetric quartic sums of squares}\label{se:quartic}
We now look at the decomposition of $H_{n,2}$ as an $\mathcal{S}_n$-module in order to apply Theorem \ref{THM Decomp} and characterize all symmetric sums of squares of degree $4$.

\begin{thm}\label{thm:decom}
Let $f^{(n)}\in H_{n,4}$ be symmetric and $n\geq 4$. If $f^{(n)}$
is a sum of squares then it can be written in the form
\begin{eqnarray*}
f^{(n)}&=&\alpha_{11}p_{(1^4)}+2\alpha_{12}p_{(2,1^2)}+\alpha_{22}p_{(2^2)}\\
&+&\beta_{11}\left(p_{(2,1^2)}-p_{(1^4)}\right)+2\beta_{12}\left(p_{(3,1)}-p_{(2,1^2)}\right)+\beta_{22}\left(p_{(4)}-p_{(2^2)}\right)\\
&+&\gamma\left(\frac{1}{2}
p_{(1^4)}-p_{(2,1^2)}+\frac{n^2-3n+3}{2n^2}p_{(2^2)}+\frac{2n-2}{n^2}
p_{(31)}+\frac{1-n}{2n^2} p_{(4)}\right)
\end{eqnarray*}
such that $\gamma\geq 0$ and the matrices $\begin{pmatrix}
\alpha_{11}&\alpha_{12}\\
\alpha_{12}&\alpha_{22}
\end{pmatrix}
$ and $\begin{pmatrix}
\beta_{11}&\beta_{12}\\
\beta_{12}&\beta_{22}
\end{pmatrix}$ are positive semidefinite.
\end{thm}
\begin{proof}
The statement follows directly from the arguments presented in Subsection \ref{sub:fixeddegree}.
Following Theorem \ref{thm:GP3} we get that  $f^{(n)}$ has a decomposition in the form
$$f^{(n)}=B^{(n)}+\langle B^{(n-1,1)},Q_n^{(n-1,1)}\rangle + B^{(n-2,2)} \cdot  Q_n^{(n-2,2)},$$
where $B^{(n)}$ is a sum of symmetric squares, $B^{(n-1,1)}$ is a $2\times 2$ sum of symmetric squares matrix polynomial and due to the degree restrictions $B^{(n-2,2)}$ is a non-negative scalar.
It remains to calculate the  matrices $Q_n^{(n-1,1)}$ and $Q_n^{(n-2,2)}$ appearing in  the statement decomposition. These are defined as the symmetrization of the pairwise products of  those Specht polynomials which generate the corresponding Specht modules in degree 2.  
In degree 2 the  Specht polynomials $X_n-X_1$ and $X_n^2-X_1^2$ generate two disjunct irreducible $\mathcal{S}_n$ modules isomorphic to $S^{(n-1,1)}$ part and the Specht polynomial $(X_{n-1}-X_1)(X_{n}-X_2)$ generates a module isomorphic to  $S^{(n-2,2)}$. Thus we have:
$$
\begin{aligned}
Q_n^{(n-1,1)}&=\begin{pmatrix} \Sym_n ((X_n-X_1)^2)& \Sym_n((X_n-X_1)(X_n^2-X_1^2)\\
                                                         \Sym_n((X_n-X_1)(X_n^2-X_1^2))&\Sym_n((X_n^2-X_1^2)^2)
                                                         \end{pmatrix}\\
                                                         Q_n^{(n-2,2)}&=\Sym_n \left((X_{n-1}-X_1)^2(X_n-X_2)^2 \right).                                      
\end{aligned}
$$

Then the symmetrisation can be calculated quite directly, since ever of the products is only involves at most $4$ variables. These calculations then yield
\begin{equation}\label{EQN SYM1}
\begin{aligned}
Q_n^{(n-1,1)}&=
                                                         \frac{2n}{n-1}\cdot \begin{pmatrix} p_{(2)}-p_{(1^2)}& p_{(3)}-p_{(2,1)}\\
                                                                                                                                      p_{(3)}-p_{(2,1)}& p_{(4)}-p_{(2^2)}\end{pmatrix},\\
                                                         Q_n^{(n-2,2)}&=  
                                                          \frac{8n^3}{n^3-6n^2+11n-6}
                                                          \left(\frac{1}{2}
p_{(1^4)}-p_{(2,1^2)}+\frac{n^2-3n+3}{2n^2}p_{(2^2)}+\frac{2n-2}{n^2}p_{(3,1)}+\frac{1-n}{2n^2} p_{(4)}\right),                                             
\end{aligned}
\end{equation}
 which gives exactly the statement in the Theorem.
\end{proof}

\subsection{The boundary of $\Sigma_{n,4}^{S}$}

We now apply Proposition \ref{COR HYPER} to the case of degree $4$ and examine the possible kernels of an extreme ray of $(\Sigma^S_{n,4})^*$ which does not come from a point evaluation.
\begin{lemma}\label{le:Kern}
Suppose a linear functional $\ell$ spans an extreme ray of
$(\Sigma_{n,4}^{\mathcal{S}})^{*}$ which is not an extreme ray of $(\mathcal{P}_{n,4}^S)^*$. Let $Q$ be quadratic form corresponding to $\ell$. Then 
\begin{equation*}
\Kern Q \simeq \mathcal{S}^{(n)}\oplus \mathcal{S}^{(n-1,1)}, 
\end{equation*}
or $$\ell \left(\sum_{\lambda \vdash 4} c_\lambda p_{\lambda}\right)=c_{(4)}+c_{(2^2)}$$
and $n$ is odd.
\end{lemma}
\begin{proof}
Since $Q$ is an $\mathcal{S}_n$-invariant quadratic form, its kernel $\Kern
Q\subseteq H_{n,2}$ is an $\mathcal{S}_n$-module. It follows from the arguments in the proof of Theorem \ref{thm:decom} that  $\Kern
Q$ decomposes as $$\Kern Q\simeq \alpha\cdot \mathcal{S}^{(n)}\bigoplus \beta\cdot
\mathcal{S}^{(n-1,1)}\bigoplus \gamma\cdot \mathcal{S}^{(n-2,2)},$$ where $\alpha,\beta\in\{0,1,2\}$ and
$\gamma\in\{0,1\}$.
We now examine the 
possible combinations of $\alpha,\beta$ and $\gamma$. 

As above let $W$ denote the kernel of $Q$.
We first observe that $\alpha=2$ is not possible: if $\alpha=2$ then we
have $p_2\in W$ which implies $p_2^2\in W^{<2>}_S$, which is
a contradiction since $p_2^2$ is not on the boundary of $\Sigma_{n,4}^{\mathcal{S}_n}$. 

By Proposition \ref{COR HYPER} the kernel $W$ of $Q$ must be maximal. Let $w \in \R^n$ be the all $1$ vector: $w=(1,\dots,1)$. We now observe that $\alpha=0$ is also not possible: if $\alpha=0$ then all forms in the kernel $W$ of $Q$ are $0$ at $w$. Therefore $\ker Q \subseteq \ker Q_{\ell_w}$ and by Proposition \ref{COR HYPER} we have $Q=\lambda Q_{\ell_w}$, which is a contradiction, since $Q$ does not correspond to point evaluation. Thus we must have $\alpha=1$.

Since we have $\dim H_{n,4}^{\mathcal{S}}=5$ from Corollary \ref{COR HYPER} we see that $\dim W^{<2>}=4$. This excludes the case $\beta=0$, since even with $\alpha=1$ and $\gamma=1$ the dimension of $W^{<2>}$ is at most $3$. Now suppose that $\beta=2$, i.e. the $\mathcal{S}_n$-module generated by $(X_1-X_2)p_{1}$ and $X_1^2-X_2^2$ lies in $W$ as well as a polynomial $q=a p_1^2+b p_2$.
We consider the symmetrizations of the five pairwise products and express these in the basis $\{p_{(4)},p_{(3,1)}, p_{(2^2)},p_{(2,1^2)}, p_{(1^4)}\}$.
  
Now the condition  $\dim W^{<2>}=4$ implies that these 5 products cannot be linearly independent and an explicit calculation of the determinant of the corresponding matrix $M$ yields  $\det M=b(a+b)$. We now examine the possible roots of this determinant. In the case when  $a=-b$ all polynomials in $W$ (even if $\gamma=1$) will be zero at $(1,\dots,1)$, which is excluded. Therefore the only possible case is $b=0$. In that case, by calculating the kernel of $M$ we see that the unique (up to  a constant multiple) linear functional $\ell$ vanishing on $W^{<2>}$ is given by
$$\ell \left(\sum_{\lambda \vdash 4} c_\lambda p_{\lambda}\right)=c_{(4)}+c_{(2^2)}.$$

We observe using \eqref{EQN SYM1} that we must have $\gamma=0$ since $\ell \left(\Sym_n (X_1-X_2)^2(X_3-X_4)^2\right)>0$ for $n \geq 4$. Now suppose that $n$ is even and let $w\in \R^n$ be given by $w=(1,\dots,1,-1,\dots,-1)$ where $1$ and $-1$ occur $n/2$ times each. It is easy to verify that for all $f \in W$ we have $f(w)=0$. Therefore it follows that $W \subseteq \ker Q_{\ell_w}$, which is a contradiction, since $W$ is a kernel of an extreme ray which does not come from point evaluation.

When $n$ is odd the forms in $W$ have no common zeroes and therefore $\ell$ is not a positive combination of point evaluations. It is not hard to verify that $\ell$ is non-negative on squares and the kernel $W$ is maximal. Therefore by Proposition \ref{COR HYPER} we know that $\ell$ spans an extreme ray of $(\Sigma^S_{n,2d})^*$

Finally we need to deal with the case $\alpha=\beta=\gamma=1$. Suppose that the $\mathcal{S}_n$-module $W$ is generated  by  three
polynomials:  $$q_1:=a p_{1}^{2}+b
p_{2}, \quad q_2:=c(X_1-X_2)p_{1}+d(X_1^2-X_2^2),\quad q_3=(X_1-X_2)(X_3-X_4).$$
Again we consider the symmetrizations of the five pairwise products and represent these in a matrix $M$ .
Explicit calculations now show that $$\det M=-(a+b)(a d^2n^2-4ad^2n+4ad^2+bd^2n^2+4bcd n+bc^2n-4bcd-bc^2).$$
Since we have $\alpha=\beta=\gamma=1$ we must have $\operatorname{rank} M=4$ since the rows of $M$ generate $W^{<2>}$. Again we cannot have $a=-b$, and thus we must have: \begin{equation}\label{EQN mess}a d^2n^2-4ad^2n+4ad^2+bd^2n^2+4bcd n+bc^2n-4bcd-bc^2=0 \end{equation}
Therefore there exists a unique linear functional $\ell$, which vanishes on $W^{<2>}$ and $\ell$ comes from the kernel of $M$. 

Let $w\in \R^n$ be a point with coordinates $w=(s,\dots,s,t)$ with $s,t \in \R$, such that they satisfy:
$$c n(s+t)+d((n-1)s+t)=0.$$
We see that $q_3(w)=0$ and from the above equation it also follows that for all $f$ in the $\mathcal{S}_n$-module generated by $q_2$ we have $f(w)=0$. Direct calculation shows that \eqref{EQN mess} also implies that $q_1(w)=0$. Thus we have $W\subseteq Q_{\ell_w}$, which is a contradiction by Proposition \ref{COR HYPER}, since $W$ is a kernel of an extreme ray which does not come from point evaluation. We remark that it is possible to show that the functional $\ell$ vanishing on $W^{<2>}$ and giving rise to $W$ is in fact a multiple of $\ell_w$, but this is not necessary for us to finish the proof.
\end{proof}
The above description allows us to explicitly characterize
degree 4 symmetric sums of squares that are positive and on the
boundary of $\Sigma_{n,4}^S$.

\begin{thm}\label{thm:boundary}
Let $n\geq 4$ and $f^{(n)}\in H_{n,4}$ be symmetric and positive and on the boundary of
$\Sigma_{4,n}^{S}$. Then
\begin{enumerate}
 \item either $f^{(n)}$ can be written as
$$f^{(n)}=a^2p_{(4)}^{(n)}+2abp_{(31)}^{(n)}+(c^2-a^2)p_{2^2}^{(n)}+(2cd+b^2-2ab)p_{(2,1^2)}^{(n)}+(d^2-b^2)p_{(1^4)}^{(n)},$$
with non-zero coefficients $a,b,c,d\in\R\backslash\{0\}$ which
additionally satisfy

\begin{equation}\label{eq:coef}
\begin{tabular}{llll}
0&$\leq\frac{a(c-d)+b(d+c)}{ac}$&0&$\leq\frac{a (c + d) (b c - a d)}{ac}$\\\\
0&$\leq-\frac{c+d}{a^2c^2}(a^2 (c - d) + b (a c + b c))$&0&$\leq-\frac{c+d}{a^2c^2}((a b c + b^2 c - a^2 d) a^2 c - (-a^2 d)^2)$\\\\
0&\multicolumn{3}{l}{$\leq \left( c+d \right)  \left(\left( c{a}^{2}+cab \right) {n}^{2}+ \left( {b}^{2}c-3\,cab+3\,{a}^{2}d \right) n-{b}^{2}c+3\,cab-3\,{a}^{2}d\right)$}\\\\
\end{tabular}
\end{equation}
\item or if $n$ is odd then  $f^{(n)}$ may have the form $$f^{(n)}=a^2p_{(1^4)}+b_{11}\left(p_{(2,1^2)}-p_{(1^4)}\right)+2b_{12}\left(p_{(3,1)}-p_{(2,1^2)}\right)+b_{22}\left(p_{(4)}-p_{(2^2)}\right),$$
with  coefficients $a,b_{11}, b_{12}, b_{22}\in\R$ which
additionally satisfy
\begin{equation*}\label{eq:coef2}
 a\neq 0,\, b_{11}+b_{22}\geq 0,\, b_{11}b_{22}-b_{12}^2\geq 0 \end{equation*}
\end{enumerate}
\end{thm}
\begin{proof}
Suppose that $f^{(n)}$ is a strictly positive form on the boundary of
$\Sigma_{4,n}^{S}$. Then
there exists a non-trivial functional $\l$ spanning an extreme ray of
the dual cone $(\Sigma_{4,n}^{S})^{*}$ such that $\ell(f^{(n)})=0$. Let
$W_{\l}\subset
H_{n,2}$ denote the kernel of $Q_\ell$. In view of Lemma \ref{le:Kern}
we see that there are two 
possible situations that we need to take into consideration.

$(1)$ We first assume that 
\begin{equation}\label{eq:kern}W_\ell\simeq \mathcal{S}^{(n)}\oplus
\mathcal{S}^{(n-1,1)}.\end{equation}
In view of \eqref{eq:kern} we may assume that the $\mathcal{S}_n$-module $W_\ell$ is generated  by  two
polynomials:  $$q_1:=(c p_{1}^{2}+d
p_{2}) \text{ and } q_2:=\frac{n-1}{2n} (a  (X_1^2-X_2^2)+b(X_1-X_2)p_{1}),$$
where $a,b,c,d\in\R$ are chosen such that $(0,0)\neq (a,b)$ and
$(0,0)\neq (c,d)$.

Let $q\in H_{n,2}$. By Proposition \ref{COR HYPER} we have$$ q\in
W_{\l}\,\,\text{ if and only if the $\mathcal{S}_n $-linear map}\,\,
p\mapsto \ell(pq)
\,\,\text{ is the zero map for on } H_{n,2}.$$
The dimension of the vector space of $\mathcal{S}_n$-invariant quadratic maps from $H_{n,2}$ to
$\R$ is 5. However, since $q\in W_{\ell}$ Schur's lemma implies
$\ell(q\cdot r)=0$ for all $r$ in the isotypic  component
of the type $(n-2,2)$.
Let $y_{\lambda}=\ell(p_{\lambda})$.  Using explicit calculations we find that the
coefficients $y_{\lambda}$ are characterized by the
following system of four linear equations:
\begin{align*}\label{eq:system}
 0=&\ell(\Sym(q_1\cdot p_{2}))=c \cdot y_{(2^2)}+d\cdot y_{(2,1^2)}\\
 0=&\ell(\Sym(q_1\cdot p_{1}^2))=c\cdot y_{(2,1^2)} +d\cdot  y_{(1^4)}\\
 0=&\ell(\Sym(q_2\cdot(X_1^2-X_2^2)))=a \cdot y_{(4)}-a\cdot
y_{(2^2)}+b\cdot y_{(3,1)}-b\cdot y_{(2,1^2)}\\
 0=&\ell(\Sym(q_2\cdot(X_1-X_2)p_{1}))=a\cdot y_{(3,1)}-a\cdot
y_{(2,1^2)}+b\cdot y_{(2,1^2)}- b\cdot y_{(1^4)}
\end{align*}

Since in addition we want that the form $\l\in(\Sigma_{n,d}^S)^{*}$ we
must also take into account that the corresponding quadratic form
$Q_\ell$ has to be positive semidefinite.  By Lemma \ref{le:psd}  this
is equivalent to checking that
each of the two matrices
$$M_{(n)}:=\begin{pmatrix}
                    y_{(2^2)}&y_{(2,1^2)}\\
                    y_{(2,1^2)}&y_{(1^4)}
                    \end{pmatrix},
                    M_{(n-1,1)}:=\begin{pmatrix}
                   y_{(4)}-y_{(2^2)}&y_{(3,1)}-y_{(2,1^2)}\\
                    y_{(3,1)}-y_{(2,1^2)}&y_{(2,1^2)}-y_{(1^4)}
                    \end{pmatrix}\text{ is positive semidefinite, }$$
                    $$\text{ and 
}M_{(n-2,2)}:=\frac{n^2}{2}y_{(1^4)}-n^2y_{(21^2)}+(2n-2)y_{(31)}+\frac{1}{2}(n^2-3n+3)y_{(2^2)}+\frac{1-n}{2}
y_{(4)}\geq 0.$$

Now assuming $a=0$ we find that, either $b=0$ which is excluded, or any
solution of the above linear system will have 
$$y_{(2^2)}=
y_{(3,1)}= y_{(2,1^2)} = y_{(1^4)}.$$ 
By substituting this into $M_{(n-2,2)}$ we find that $$\frac{1-n}{2}(y_{(4)}-y_{(2)^2}) \geq 0,$$
while from $M_{(n-1,1)}$ we have that $y_{(4)}-y_{(2)^2} \geq 0$. It follows that 
$$y_{(3,1)}= y_{(2,1^2)} = y_{(1^4)}=y_{(4)}=y_{(2)^2}.$$
But then we find that $\ell$ is proportional to the functional that simply evaluates at the point $(1,1,1,\ldots,1)$, which is a contradiction since $f^{(n)}$ is strictly positive. Thus $a\neq 0$.

Now suppose that $c=0$. Then we find that $$y_{(3,1)}=y_{(1^4)}=y_{(2,1^2)}=0 \quad \text{and} \quad a(y_{(4)}-y_{(2^2)})=0.$$

Since $a\neq 0$ we find that the linear functional $\ell$ is given by $$\ell \left(\sum_{\lambda \vdash 4} c_{\lambda}p_\lambda \right)=c_{(4)}-c_{(2^2)}.$$ By Lemma \ref{le:Kern} we must have $n$ odd in order for $\ell$ not to be a point evaluation and this lands us in case $(2)$ discussed below.

Meanwhile with $a,c \neq 0$ the solution of the linear system (up to a common multiple) is given by
$$ y_{(4)}=\frac{-b^2cd-b^2c^2+a^2d^2}{c^2a^2},\,
y_{(2^2)}=-\frac{da-db-bc}{ca},\, y_{(3,1)}= \frac{d^2}{c^2},\,
y_{(2,1^2)}=-\frac{d}{c},\, y_{(1^4)}=1,$$
which then yields the conditions in \eqref{eq:coef}.
 
$(2)$ If $n$ is odd we know from Lemma 6.7 that there is one additional case: $f^{(n)}$ is a sum of the square $(ap_{(11)})^2$, and a sum of squares of elements from the isotypic component
of $H_{n,2}$ which corresponds to the representation $S^{(n-1,1)}$. Since $f^{(n)}$ is strictly positive, we must have $a\neq0$ (otherwise $f^{(n)}$ has a zero at $(1,\dots,1)$) and it also follows that the matrix 
$\begin{pmatrix}
b_{11}&b_{12}\\
b_{12}&b_{22}
\end{pmatrix}$
must be strictly positive definite.
Therefore we get the announced decomposition from Theorem \ref{thm:decom}.

 \end{proof}

Note that although the first symmetric counterexample by Choi and Lam in four
variables gives $\Sigma_{4,4}^S\subsetneq \mathcal{P}_{4,4}^S$ it does
not immediately imply that we have strict containment for all $n$. 

However, using our methods, one can 
produce a sequence of strictly positive symmetric quartics that
lie on the boundary of $\Sigma_{n,4}^{S}$ for all $n$ as a witness for
the strict inclusion. 

\begin{ex}\label{ex:nice}

For $n\geq 4$ consider family of  polynomials
$$f^{(n)}:=a^2p_{(4)}^{(n)}+2abp_{(31)}^{(n)}+(c^2-a^2)p_{(2^2)}^{(n)}+(2cd+b^2-2ab)p_{(2,1^2)}^{(n)}+(d^2-b^2)p_{(1^4)}^{(n)},$$
where we set $a=1,{ b=-\frac{13}{10}}, c=1$ and $d=-\frac{5}{4}$.
Further consider the linear functional $\l\in H_{n,4}^{*}$ with
\begin{equation*}\l\left(p_{(4)}^{(n)}\right)={\frac{397}{200}},~
\l\left(p_{(2^2)}^{(n)}\right)={\frac{63}{40}},
~\l\left(p_{(3,1)}^{(n)}\right)={\frac{25}{16}},
~\l\left(p_{(2,1^2)}^{(n)}\right)=\frac{5}{4},~
\l\left(p_{(1^4)}^{(n)}\right)=1.\end{equation*} Then we have
$\ell(f^{(n)})=0$. In addition the corresponding matrices become
$$M_{(n)}:=\left( \begin {array}{cc} {\frac {63}{40}}&\frac{5}{4}\\ \noalign{\medskip}\frac{5}{
4}&1\end {array} \right) ,M_{(n-1,1)}:=  \left(\begin {array}{cc} {\frac {41}{100}}&{\frac {5}{16}}
\\ \noalign{\medskip}{\frac {5}{16}}&\frac{1}{4}\end {array} \right), \text{ and }
M_{(n-2,2)}:={\frac {3}{80}}\,{n}^{2}+{\frac {21}{80}}\,n-{\frac {21}{80}}.$$ These matrices are all positive semidefinite for $n\geq 4$ and
therefore we have $\l\in\left(\Sigma_{n,4}^{S}\right)^{*}$. This
implies that  $f^{(n)}\in\partial\Sigma_{n,4}^{S}$.

\noindent Now we argue that for any $n\in \N$ the forms $f^{(n)}$ are
strictly positive. By Corollary \ref{cor:degree} it follows that
$f^{(n)}$ has a zero, if and only if there exists
$k\in\{\frac{1}{n},\ldots,\frac{n-1}{n}\}$ such that the bivariate
form
\begin{align*}h_k(x,y)=\Phi_f(k,1-k,x,y)&=k{x}^{4}+(1-k)y^4-{\frac {13}{5}}\, \left( k{x}^{3}+(1-k)y^3 \right)  \left( kx+
(1-k)y \right)\\& +{\frac {179}{100}}\, \left( k{x}^{2}+(1-k)y^2 \right)  \left( 
kx+(1-k)y \right) ^{2}-{\frac {51}{400}}\, \left( kx+(1-k)y \right) ^{4}\end{align*} has a real projective zero $(x,y)$. 

Since $f^{(n)}$ is a sum of squares and therefore
non-negative we also know that $h_k(x,y)$ is non-negative for all $k \in \{\frac{1}{n},\ldots,\frac{n-1}{n}\}$.
Therefore the real projective roots of  $h_k(x,y)$ must have even multiplicity.
This implies that $h_k(x,y)$ has a real root  only if its discriminant
$\delta(h_k)$ - viewed as polynomial in the parameter $k$ - has a root
in the admissible range for $k$.
 We calculate \begin{equation*}\delta(h_k):=-{10^{-8}}\, \left( 10000-37399\,k+37399\,{k}^{2}
 \right)  \left( 149\,{k}^{2}-149\,k+25 \right) ^{2} \left( k-1
 \right) ^{3}{k}^{3}
\end{equation*}
We see that $\delta(h_k)$ is zero only for $k\in\{0,1,
\frac{1}{2}\pm{\frac {7}{298}}\,\sqrt {149}, \frac{1}{2}\pm{\frac {51}{74798}}\,i\sqrt {37399}\}$. 
Thus we see that for all natural numbers $n$ there is no
$k\in\{\frac{1}{n},\ldots,\frac{n-1}{n}\}$ such that $h_k(x,y)$ has a
real projective zero. Therefore we can conclude that for any
$n\in \N$ the form $f^{(n)}$ will be strictly positive. 
\end{ex}
From the above example the following characterization, which recently had been independently given by Goel, Kuhlmann and Reznick  \cite{gkr}  is an  immediate consequence.
\begin{thm}\label{thm:1888}
The inclusion $\Sigma_{n,2d}^S\subset \mathcal{P}_{n,2d}^S$ is strict except in the cases of symmetric bivariate form, or symmetric quadratic forms, or symmetric ternary quartics.
\end{thm}
\begin{proof}
The well-known Robinson form $$X_1^6 + X_2^6 + X_3^6 - X_1^4X_2^2-X_1^2X_2^4 - X_1^4X_3^2- X_2^4X_3^2 - X_1^2X_3^4 -X_2^2X_3^4 + 3X_1^2X_2^2X_3^2$$ is a non-negative form which is not a sum of squares.
Furthermore, for the case $2d=4$ and $n\geq 4$ Example \ref{ex:nice} above  gives for every $n$ a positive polynomial $f^{(n)}$ which lies on the boundary of $\Sigma_{n,4}^S$ and therefore guarantees the existence of $h_{n,4}\in P_{n,2d}^S$ which is positive definite but not a sum of squares. The result now follows by observing that  for any  positive definite form $h\in H_{n,2d}$ that is not a sum of squares, the form $(X_1+\ldots+ X_n)^2 h\in H_{n,2d+2}$ is also positive definite and not a sum of squares. Indeed, suppose that  $(X_1+\ldots+ X_n)^2 h=f_1^2+\ldots+f_m^2$ then $(X_1+\ldots+ X_n)^2$ will divide $f_i^2$ which yields that $h$ is a sum of squares. 
\end{proof}

\section{Asymptotic  behavior }\label{seq:limit}
In this section we  study the relationship of sums of squares and non-negative forms when the number of variables tends to infinity.

\subsection{Full dimensionality}

We now consider the power mean inequalities and their limits.  In order to talk about limits of our sequences of cones we use following notion of limit of a sequence of sets, which is due to Kuratowski \cite{kur} and we refer the reader to \cite{mos,sal} for details in the context of sequences of convex sets.
\begin{definition}\label{def:kur}
Let $\{K_n\}_{n\in N}$ be sequence of subsets of $\R^k$. Then a set $K\subset\R^k$ is called the limit of the sequence, denoted by $K=\lim_{n\rightarrow \infty} K_n$, if we have
$$\limsup_{n\rightarrow \infty} K_n\subset K\subset \liminf_{n\rightarrow \infty} K_n,$$ where
\begin{align*}
\liminf_{n\rightarrow \infty} K_n&=\{x\in \R^k\,:\, x=\lim_{n\in \N} X_n\, , X_n\in K_n\}\\
\\
\limsup_{n\rightarrow \infty} K_n&=\{x\in \R^k\,:\, x=\lim_{m\in M\subset \N} X_m\, , X_m\in K_m\}, \text{ for some infinite }  M\subset \N.\\
\end{align*}
\end{definition}
\begin{remark}
Note that the limit defined above is a closed set.
\end{remark}

It will be convenient for the proof Theorem \ref{THM LIMIT} to relate the power mean inequalities to the sequences formed by the Reynolds operator. Let $\mu=(\mu_1,\dots,\mu_r)$ be a partition of $2d$. Associate to $\mu$ the monomial $X_1^{\mu_1}\cdots X_{r}^{\mu_r}$ and define symmetric form $m_\mu^{(n)}$ by:
$$m_\mu^{(n)}=\Sym_n (X_1^{\mu_1}\cdots X_{r}^{\mu_r}).$$
This is the monomial mean basis of $H_{n,2d}^S$. We observe that with this choice of basis $H_{n,2d}^S$ the transition maps $\varphi_{m,n}$ are given by the identity matrices. Since the stabilizer of the monomial $X_1^{\mu_1}\cdots X_{r}^{\mu_r}$ is isomorphic to $\mathcal{S}_{s_1}\times \ldots\times \mathcal{S}_{s_t}\times \mathcal{S}_{n-r}$ it follows that
$$m_\mu^{(n)}=\frac{s_1!\cdots s_k!(n-r)!}{n!}\bar{m}_{\mu}^{(n)}=\binom{n}{s_1\dots s_k}^{-1}\bar{m}_{\mu}^{(n)},$$
where $\bar{m}_{\mu}^{(n)}$ is the monomial symmetric polynomial associated to $\mu$.

\begin{prop}\label{prop:limmon}
Consider the sequences $\Sigma_{n,2d}^{\varphi}$ and $\mathcal{P}^{\varphi}_{n,2d}$ embedded into $\R^{\pi(2d)}$ via the monomial mean basis. Then the limits of the resulting sequences of convex cones in $\R^{\pi(2d)}$ have limits, which we will denote by $\mathfrak{S}_{2d}^{\varphi}$ and $\mathfrak{P}_{2d}^{\varphi}$. Both of these limits are closed and full-dimensional.
\end{prop}
\begin{proof}
Since we have $\varphi_{n,n+1}(\Sigma_{n,2d})\subseteq \Sigma_{n+1,2d}^{S}$ and $\varphi_{n,n+1}(P_{n,2d})\subseteq \mathcal{P}^{S}_{n+1,2d}$ the resulting sequences of cones are increasing. Thus by \cite{sal}[Proposition 1] the limits exist and are given by 
\begin{eqnarray*}
\mathfrak{S}_{2d}^{\varphi} &=&\overline{\left\{f^{(n)}:=\sum_{\lambda\vdash 2d} c_\lambda m_{\lambda}^{(n)} \text{ with } f^{(m)}\in \Sigma_{m,2d}, \text{for one } m\in \N\right\}}\\
\mathfrak{P}_{2d}^{\varphi}&=&\overline{\left\{f^{(n)}:=\sum_{\lambda\vdash 2d} c_\lambda m_{\lambda}^{(n)} \text{ with } f^{(m)}\in \mathcal{P}_{m,2d}, \text{for one } m\in \N\right\}}
\end{eqnarray*}
Clearly, both cones are full dimensional.
\end{proof}

In order to establish the result for the power mean basis, we first have to study the relationship between these two bases:
\begin{prop}\label{prop:transition}
Let $M_n$ be the matrix converting between the monomial mean and power mean basis of $H_{n,2d}^S$. Then $M_n$ converges entry-wise to a full rank matrix $M^*$ as $n$ grows to infinity.
\end{prop}
\begin{proof}
The transition matrix between power sum symmetric polynomials and monomial symmetric polynomials is well understood \cite{Transition}. Converting to our mean bases we have the following: let $\nu:=(\nu_1,\ldots,\nu_l)\vdash 2d$, $\mu=(\mu_1,\dots,\mu_r) \vdash 2d$, then: 
$$m_\mu^{(n)}=\sum_{\nu\vdash 2d}(-1)^{r-l} \frac{(n-r)!|\mathcal{BL}(\mu)^{\nu}|}{n!}n^{l}p_\nu^{(n)},$$
where $|\mathcal{BL}(\mu)^{\nu}|$ is the number of $\mu$-brick permutations of shape $\nu$ \cite{Transition}. We observe that the unique highest order of growth in $n$ for a coefficient of $p_{\nu}^{(n)}$ occurs when the number of parts of $\nu$ is maximized. The unique $\nu$ with the largest number of parts and nonzero $|\mathcal{BL}(\mu)^{\nu}|$ is $\mu$. Thus we have $\nu=\mu$, $r=l$, and
$$|\mathcal{BL}(\nu)^{\nu}|=\nu_1!\cdots \nu_l!  \quad \text{and} \quad \lim_{n \rightarrow \infty} \frac{n^r(n-r)!}{n!}=1.$$
Therefore we see that asymptotically $$m_{\mu}^{(n)}=p_{\mu}^{(n)}+\sum_{\nu \vdash 2d, \nu \neq \mu} a_{\mu,\nu}(n)p_{\nu}^{(n)},$$
where the coefficients $a_{\mu,\nu}(n)$ tend to $0$ as $n \rightarrow \infty$. The Proposition now follows.
\end{proof}
 Now with these preparations the  proof of Theorem \ref{THM LIMIT} will be immediate after the following  two Lemmata.
\begin{lemma}\label{lemma:limits}
Let $V$ be a finite dimensional vector space. Let $A_i$ be a sequence of subsets of $V$ converging to $A$. Let $M_i$ be a sequence of linear maps from $V$ to itself converging to identity. Let $B_i=M_i(A_i)$. Then
$$\operatorname  A \subseteq \liminf_{i \rightarrow \infty} B_i \hspace{.5cm} \text{and} \hspace{.5cm} \limsup_{i \rightarrow \infty} B_i \subseteq A,$$ so $A$ is the limit of $B_i$.
\end{lemma}

\begin{proof}
For the first inclusion, let $a\in  A$. Since $A$ is the limit of $A_i$ there exists $N$ such that for all $i\geq N$ $a$ is contained in $A_i$. Let $b_i=M_i a$. Then $b_i\in B_i$  for all $i\geq N$ and  moreover, since the linear maps $M_i$ converge to identity we have that $b_i$ converges to $a$. This in turn implies that $a\in \liminf_{i \rightarrow \infty} B_i$

\indent For the second inclusion, we remark that $(\limsup_{i \rightarrow \infty} B_i)^c= \liminf_{i \rightarrow \infty} B_i^c$. Hence,  one can argue in an analogous way by considering the complement of $A$.
\end{proof}

From the above lemma we can easily obtain the following generalization, which shows that the conclusions also hold if the limit of the linear maps $M_i$ is any full-rank map. 

\begin{lemma}\label{lemma:limits2}
Let $V$ be a finite dimensional vector space. Let $A_i$ be a sequence of subsets of $V$ converging to $A$. Let $M_i$ be a sequence of linear maps from $V$ to itself, converging to a full-rank linear map $M$. Let $B_i=M_i(A_i)$. Then

$$M(A) \subseteq \liminf_{i \rightarrow \infty} B_i \hspace{.5cm} \text{and} \hspace{.5cm} \limsup_{i \rightarrow \infty} B_i \subseteq M(A),$$
so $M(A)$ is the limit of $B_i$.
\end{lemma}
\begin{proof}
We can apply Lemma \ref{lemma:limits} to the sequence $C_i=M^{-1}M_i(A_i)$. Since $B_i=M(C_i)$, and $M$ is a full-rank linear map, the desired conclusions follow for the sequence $B_i$ as well.
\end{proof}
The existence of the  limits of sequences and their full dimensionality now can be established by translating from Proposition \ref{prop:limmon}.
\begin{proof}[Proof of Theorem \ref{THM LIMIT}]
We only give the proof for $\mathfrak{S}_{2d}$ since the statement for $\mathfrak{P}_{2d}$ follows in an analogous manner. 
We first observe that  the sequence $\Sigma_{n,2d}^\rho$ is semi nested it follows that $\liminf \Sigma_{n,2d}^\rho=\bigcap_{n\geq 2d} \Sigma_{n,2d}^\rho$. We now apply Lemma \ref{lemma:limits} to the sequence $\Sigma_{n,2d}^\rho$, with $A_i=\Sigma^{\varphi}_{n,2d}$ and $M_i$ transition maps between monomial mean and power mean bases. From Proposition \ref{prop:transition} we know that the maps $M_i$ converge to identity. Therefore, we see that $$\mathfrak{S}_{2d}\subseteq \bigcap_{n\geq 2d} \Sigma_{n,2d}^\rho\text{ and } \limsup \Sigma_{n,2d}^\rho \subseteq \mathfrak{S}_{2d}.$$ The theorem now follows, since the full-dimensionality is a direct consequence of Proposition \ref{prop:limmon}.
\end{proof}

\section{Symmetric mean inequalities of degree four}\label{se:main}
In this last section we  characterize quartic symmetric mean
inequalities that are valid for all values of $n$. Recall from 
Section 2  that  $\mathfrak{P}_4$ denotes the cone of all sequences
$\mathfrak{f}=(f^{(4)},f^{(5)},\ldots)$ of degree 4 power means that
are non-negative for all $n$ and  $\mathfrak{S}_4$ the cone of such
sequences that can be written as sums of squares.

In the case of quartic forms  the elements of $\mathfrak{P}_{4}$ can
be characterized  by a family of univariate polynomials as Theorem
\ref{thm:charpos} specializes to the following

\begin{prop}\label{le:ph1}
Let
$$\mathfrak{f}:=\sum_{\lambda \vdash
4}c_{\lambda}\mathfrak{p}_{\lambda}$$ be a linear combination
of quartic symmetric power means.
Then $\mathfrak{f}\in\mathfrak{P}_4$ if any only if for all $\alpha
\in[0,1]$ the bivariate form 
$$\Phi^{\alpha}_{\mathfrak{f}}(x,y)=\Phi_{\mathfrak{f}}(\alpha,1-\alpha,x,y):=\sum_{\lambda\vdash
 4}c_{\lambda}\Phi_{\lambda}(\alpha,1-\alpha,x,y)$$ is non-negative.
\end{prop}

Now we turn to the  characterization of the elements on the boundary
of $\mathfrak{P}_4$.
\begin{lemma}\label{le:rand}
Let $0\neq\mathfrak{f}\in\mathfrak{P}_4$. Then
$\mathfrak{f}$  is on the boundary $\partial\mathfrak{P}_4$ if and
only if  there exists  $\alpha\in(0,1)$ such that the bivariate form $\Phi^{\alpha}_{\mathfrak{f}}(x,y)$ has a double real root.
\end{lemma}
\begin{proof}
Let $\mathfrak{f}\in\partial\mathfrak{P}_4$. Suppose that for all
$\alpha\in(0,1)$ the bivariate form
$\Phi_{\mathfrak{f}}^\alpha$ has no double real roots.
From Proposition \ref{le:ph1} we know that
$\Phi_{\mathfrak{f}}^\alpha$ is a non-negative
form for all $\alpha \in [0,1]$ and thus $\Phi_{\mathfrak{f}}^\alpha$ is strictly positive for all $\alpha \in (0,1)$. Thus that for a sufficiently small
perturbation $\tilde{\mathfrak{f}}$ of the coefficients $c_{\lambda}$ of $\mathfrak{f}$
form $\Phi_{\tilde{\mathfrak{f}}}^\alpha$ will remain positive for all $\alpha \in (0,1)$. Now we deal with the cases $\alpha=0,1$.

We observe that for all $\mathfrak{g} \in H^\rho_{\infty,4}$ we have
$$\Phi^0_\mathfrak{g}(x,y)=\Phi^{1/2}_\mathfrak{g}(y,y)=\Phi^{1/2}_\mathfrak{g}(1,1)y^4 \quad \text{and} \quad \Phi^1_\mathfrak{g}(x,y)=\Phi^{1/2}_\mathfrak{g}(1,1)X^4.$$
By the above we must have $\Phi^{1/2}_\mathfrak{g}(1,1)>0$ and the same will be true for a sufficiently small
perturbation $\tilde{\mathfrak{f}}$ of $\mathfrak{f}$.
But then it follows by Proposition \ref{le:ph1} that a neighborhood of
$\mathfrak{f}$ is in $\mathfrak{P}_4$, which contradicts the
assumption that $\mathfrak{f}\in\partial\mathfrak{P}_4$. Therefore
there exists $\alpha\in(0,1)$ such that
$\Phi^\alpha_{\mathfrak{{f}}}(x,y)$ has a
double real root.

Now suppose $\mathfrak{f} \in \mathfrak{P}_4$ and $\Phi_{\mathfrak{f}}^\alpha(x,y)$ has a double real root for some $\alpha \in (0,1)$. Let $\mathfrak{f}_{\epsilon}=f-\epsilon \mathfrak{p}_{2^2}$. It follows that for all $\epsilon>0$ we have $\mathfrak{f}_\epsilon \notin \mathfrak{P}_4$, since $\Phi_{\mathfrak{f}_\epsilon}^{\alpha}$ is  strictly negative at the double zero of $\Phi_{\mathfrak{f}}^{\alpha}$. Thus $\mathfrak{f}$ is on the boundary of $\mathfrak{P}_4$.

\end{proof}

We now deduce the following Corollary from Theorem \ref{thm:decom} completely describing
polynomials belonging to $\mathfrak{S}_4$:
\begin{cor}\label{cor:SOS4}
We have $\mathfrak{f}\in\mathfrak{S}_4$ if and only if
$$\mathfrak{f}=\alpha_{11}
\mathfrak{p}_{(1^4)}+2\alpha_{12}\mathfrak{p}_{(2,1^2)}+\alpha_{22}\mathfrak{p}_{(2^2)}+\beta_{11}\left(\mathfrak{p}_{(2,1^2)}-\mathfrak{p}_{(1^4)}\right)+2\beta_{12}\left(\mathfrak{p}_{(3,1)}-\mathfrak{p}_{(2,1^2)}\right)+\beta_{22}\left(\mathfrak{p}_{(4)}-\mathfrak{p}_{(2^2)}\right),$$
where the matrices $\begin{pmatrix}
\alpha_{11}&\alpha_{12}\\
\alpha_{12}&\alpha_{22}
\end{pmatrix}
$ and $\begin{pmatrix}
\beta_{11}&\beta_{12}\\
\beta_{12}&\beta_{22}
\end{pmatrix}$
are positive semidefinite.
.
\end{cor}
\begin{proof}
We observe from Theorem \ref{thm:decom} that the coefficients of the squares of symmetric polynomials
and of  $(n-1,1)$ semi-invariants do not depend on $n$. Thus the cone generated by these sums of squares
is the same for any $n$, and it corresponds precisely to the cone given in the statement of the Corollary. Now observe that the limit of the square of
the $(n-2,2)$ component is equal to
$\frac{1}{2}\mathfrak{p}_{(1^4)}-\mathfrak{p}_{(21^2)}+\frac{1}{2}\mathfrak{p}_{(2^2)}$,
which is a sum of symmetric squares. Thus the squares from the $(n-2,2)$ component do not contribute anything in the limit.
\end{proof}

In order to algebraically characterize the elements on the boundary
recall that the discriminant $\text{disc}(f)$ of a bivariate form $f$ is a homogeneous
polynomial in the coefficients of $f$, which vanishes exactly on the set of
forms with multiple projective roots.

However, note that $\text{disc}(f)=0$
alone  does not guarantee that $f$ has a double real root, since the
double root may be complex.

\begin{prop}
Let $\mathfrak{f}\in H^\rho_{\infty,4}$ be of the from $$\mathfrak{f}=a^2\mathfrak{p}_{(1^4)}+b_{11}\left(\mathfrak{p}_{(2,1^2)}-\mathfrak{p}_{(1^4)}\right)+2b_{12}\left(\mathfrak{p}_{(3,1)}-\mathfrak{p}_{(2,1^2)}\right)+b_{22}\left(\mathfrak{p}_{(4)}-\mathfrak{p}_{(2^2)}\right),$$ such that the  coefficients meet the conditions in \eqref{eq:coef2}. Then, for $\alpha=\frac{1}{2}$ the associated form $\Phi_{\mathfrak{f}}^{\alpha}(X,Y)$ has a double root at $(x,y)=(1,-1)$.
\end{prop}

\begin{lemma}\label{le:calc}
Let $\mathfrak{f}\in H^\rho_{\infty,4}$ be of the from $$\mathfrak{f}=a^2\mathfrak{p}_4+2ab\mathfrak{p}_{31}+(c^2-a^2)\mathfrak{p}_{22}+(2cd+b^2-2ab)\mathfrak{p}
_{211}+(d^2-b^2)\mathfrak{p}_{1111},$$ such that the coefficients $a,b,c,d$
meet the conditions in \eqref{eq:coef}. Consider the  associated form $\Phi_{\mathfrak{f}}^{\alpha}$. Then, there is a value $\alpha$, with $0<\alpha<1$ such that  $\Phi_{\mathfrak{f}}^{\alpha}$ has a real double root. 
\end{lemma}
\begin{proof}
We first show that there is a (possibly complex) double root by examining the the discriminant. To this end, we find that this discriminant $\delta_{\mathfrak{f}}(\alpha)$ factors as
$$\delta_{\mathfrak{f}}(\alpha)=16(\alpha-1)^3(c+d)^2\alpha^3\delta_1(\alpha)\delta_2(\alpha)^2
,$$ 
where  $\delta_1$ and $\delta_2$ are quadratic polynomials in $\alpha$. We examine these factors $\delta_1$ and $\delta_2$ now assuming the conditions on $a,b,c,d$ imposed by \eqref{eq:coef}.

One easily checks  that $\delta_1(\alpha)=\delta_1(1-\alpha)$
  and
 $\delta_1(0)=\delta_1(1)=-16\,{a}^{2} \left( c+d \right) ^{2}<0.$ Further, $$\delta_1\left(\frac{1}{2}\right)=-\frac{1}{4}\, \left(
4\,{a}^{2}+4\,ab+4\,cd+4\,{c}^{2}+{b}^{2} \right) ^{2}<0.$$
 Which clearly implies that the quadratic polynomial $\delta_1$ is strictly negative on  $(0,1)$. Moreover, the conditions in  \eqref{eq:coef} yield 
 $\delta_2(0)=\delta_2(1)=a^2(c+d) >0$ and 
 $\delta_2\left(\frac{1}{2}\right)=\frac{1}{4}\,c \left( 2\,a+b \right) ^{2} <0 $
since $c$ is supposed to be negative.

 It follows now that 
$\delta_2(\alpha)$ has two real roots  $\alpha_{1},
1-\alpha_{1}\in(0,1)$ and hence  the polynomial
$\Phi_{\mathfrak{f}}(\alpha_{1/2},1-\alpha_{1/2},x,1)$ has a double
root.
However, it remains to verify that this double root is indeed real.
In order to establish this we examine  the polynomial $\delta_2(\alpha,a,b,d,c)$ more carefully.
We have  
$$ \delta_2(\alpha,a,b,c,d):={a}^{2}d+{a}^{2}c+4\,{\alpha}^{2}{a}^{2}d-4\,{a}^{2}\alpha
d-4\,ab{\alpha}^{2}c+4\,ab\alpha
c-{\alpha}^{2}{b}^{2}c+{b}^{2}\alpha c, 
$$
and for $\alpha\neq \frac{1}{2}$, one can solve for $d$ to find
$$d=-{{c \left( {a}^{2}-4\,ab{\alpha }^{2}+4\,ab\alpha -{\alpha
}^{2}{b}^{2}+{b}^{2}
\alpha  \right) }\left( 2a\,\alpha -a \right)^{-2}}.$$
This yields that
$\Phi_{\mathfrak{f}}(\alpha^{*},1-\alpha^{*},x,1)$ contains the factor
$\left( ax+a+b\alpha x-b\alpha +b \right) ^{2}$ and hence in this case
 $\Phi_{\mathfrak{f}}(\alpha^{*},1-\alpha^{*},x,1)$ has
real double root.

In the case $\alpha=\frac{1}{2}$ it follows from the observations made above, that at a root of $\delta_2(\alpha,a,b,c,d)$ also all second partial derivatives have to vanish.
By explicit calculations one finds that this can happen only if $a=-\frac{1}{2}$ in which case the polynomial $\delta_2$
specializes to $1/4\,c \left( -1+b \right) ^{2}$.
Since $c<0$ it follows that only for $b=1$ the discriminant can vanish.
In this situation one gets
$$\Phi_{\mathfrak{f}}(\frac{1}{2},\frac{1}{2},x,1)=\frac{1}{16}\,
\left( d+2\,dx+2\,c+2\,c{x}^{2}+d{x}^{2} \right) ^{2}$$
and it follows that  $X_{1/2}:=\pm { {-d+2\,\sqrt
{-{c}(c-d)}}{(2\,c+d){-1}}}$ are the two double roots in this case. The  conditions posed on $c,d$ by Theorem \ref{thm:boundary}
ensure that  $-{c}(c-d)\geq 0$, and hence these roots will also be real.
Therefore we have shown that in all cases the roots are indeed real.

\end{proof}
We are now in the position to show that $\mathfrak{P}_4=\mathfrak{S}_4$.
\begin{proof}[Proof of Theorem \ref{THM DEG4}]
Since $\mathfrak{S}_4\subset\mathfrak{P}_4$ and both sets are closed
convex cones, it suffices to show that  every $\mathfrak{f}$ on the boundary of $\mathfrak{S}_4$ also lies
in the boundary of $\mathfrak{P}_4$. It follows from  Theorem \ref{thm:boundary} that a sequence
$\mathfrak{f}:=(f^{(4)},f^{(5)},\ldots)$ in the boundary of 
$\mathfrak{S}_4$ which is not in the boundary of $\mathfrak{P}_4$ has to be a form as considered in Lemma \ref{le:calc}, but by combining Lemmas  \ref{le:calc} and \ref{le:rand} we find that $\mathfrak{f}\in\partial\mathfrak{S}_{4}$ implies that $\mathfrak{f}\in\partial\mathfrak{P}_4$, and we can conclude that
$\mathfrak{S}_4=\mathfrak{P}_4$.
\end{proof}

\section{Conclusion, open questions and acknowledgements}

Besides Conjecture \ref{CONJ MAIN}  there is another important question left open in our work. Corollary \ref{cor:SOS4} gave a description of the asymptotic
symmetric sums of squares cone in terms of the squares involved. In this description of the limit
not all semi-invariant polynomials were necessary. It is natural  to
investigate the situation also in arbitrary degree:

\begin{con}
Let $\mathfrak{f}\in\mathfrak{S}_{2d}$. What semi-invariant
polynomials are necessary for a description of $\mathfrak{f}$ as a sum
of squares?
\end{con}

The general setup of our work focused on the case of a fixed degree. Examples like the difference of the geometric and the arithmetic mean show however, that it would be very interesting to also understand the situation where the degree is not fixed.
\begin{con}
What can be said about the quantitative relationship between the cones $\Sigma^S_{n,2d}$ and $\mathcal{P}^S_{n,2d}$ in asymptotic regimes other than fixed degree $2d$?
\end{con}
 
This research was initiated  during the IPAM
program on
Modern Trends in Optimization and its Application and the authors would like
to thank the Institute for Pure and Applied Mathematics for the hospitality
during the program and the organizers of the program for the
invitation to participate int the program. We thank an anonymous referee  for helpful comments that greatly improved this paper and Roland Hildebrand for bringing the article by Terpstra to our attention. The second author acknowledges support of  the Troms\o~ Research foundation und grant agreement 17matteCRMarie.

\end{document}